\documentclass{article}%
\usepackage{amsmath}
\usepackage{amsfonts}
\usepackage{amssymb}
\usepackage{graphicx}%
\usepackage{mathrsfs}
\providecommand{\U}[1]{\protect \rule{.1in}{.1in}}
\newtheorem{theorem}{Theorem}[section]

\newtheorem{lemma}[theorem]{Lemma}

\newtheorem{remark}[theorem]{Remark}

\newenvironment{proof}[1][Proof]{\noindent \textbf{#1.} }{\  \rule{0.5em}{0.5em}}
\begin{document}

\title{A type of globally solvable BSDEs with triangularly quadratic generators}
\author{Peng Luo \thanks{Department of Statistics and Actuarial Science, University of Waterloo, Waterloo, ON,
N2L 3G1, Canada (p23luo@uwaterloo.ca)}}

\maketitle

\begin{abstract}
The present paper is devoted to the study of the well-posedness of a type of BSDEs with triangularly quadratic generators. This work is motivated by the recent results obtained by Hu and Tang \cite{HT} and Xing and \v{Z}itkovi\'{c} \cite{XZ}. By the contraction mapping argument, we first prove that this type of triangularly quadratic BSDEs admits a unique local solution on a small time interval whenever the terminal value is bounded. Under additional assumptions, we build the global solution on the whole time interval by stitching local solutions. Finally, we give solvability results when the generators have path dependence in value process.
\end{abstract}


\textbf{MSC-classification}: 60H10, 60H30.
\section{Introduction}
Backward stochastic differential equations (BSDEs) are introduced in Bismut \cite{B}.
A BSDE is an equation of the form
\begin{equation*}
Y_t=\xi+\int_{t}^{T}g(s,Y_s,Z_s)ds-\int_{t}^{T}Z_sdW_s, \quad t\in[0,T],
\end{equation*}
where $W$ is a $d$-dimensional Brownian motion,  the terminal condition $\xi$ is an $n$-dimensional random variable, and $g:\Omega\times [0,T]\times \mathbb{R}^n \times \mathbb{R}^{n\times d}\to \mathbb{R}^n$ is the generator. A solution consists of a pair of predictable processes $(Y,Z)$ with values in $\mathbb{R}^n$
and $\mathbb{R}^{n\times d}$, called the value and control process, respectively. The first existence and uniqueness result for BSDEs with an $L^2$-terminal condition and a generator satisfying a Lipschitz growth condition is due to  Pardoux and Peng~\cite{PP}. In case that the generator satisfies a quadratic growth condition in the control $z$, the situation is more involved and a general existence theory  does not exist. Frei and dos Reis \cite{FR} and Frei \cite{F} provide counterexamples which show that multidimensional quadratic BSDEs may fail to have a global solution. In the one-dimensional case the existence of quadratic BSDE
is shown by Kobylanski \cite{Ko} for bounded terminal conditions, and by Briand and Hu \cite{BH,BH1} for unbounded terminal conditions. Briand and Elie \cite{BE} provide a constructive approach to quadratic BSDEs with and without delay. Solvability results for superquadratic BSDEs are discussed in Delbaen et al.~\cite{DHB}, see also Masiero and Richou \cite{MR}, Richou \cite{R} and Cheridito and Nam  \cite{CN1}.

The focus of the present work lies on multidimensional quadratic BSDEs. In case that the terminal condition is small enough the
existence and uniqueness of a solution was first shown by Tevzadze \cite{T}. Cheridito and Nam \cite{CN} and Hu and Tang \cite{HT} obtain local solvability on $[T-\varepsilon,T]$ for some $\varepsilon>0$ of systems of BSDEs with subquadratic generators and diagonally quadratic generators respectively, which under additional assumptions on the generator can be extended to global solutions.
Cheridito and Nam \cite{CN} provide solvability for Markovian quadratic BSDEs and projectable quadratic BSDEs. Xing and \v{Z}itkovi\'{c} \cite{XZ} obtained the global solvability for a large class of multidimensional quadratic BSDEs in the Markovian setting. Frei \cite{F} introduced the notion of split solution and studied the existence of solution for multidimensional quadratic BSDEs by considering a special kind of terminal condition. Jamneshan et al.~\cite{JKL} provide solutions for multidimensional quadratic BSDEs with separated generators. Using a stability approach, Harter and Richou \cite{HR} establish an existence and uniqueness result for a class of multidimensional quadratic BSDEs. In Bahlali et al.~\cite{BEH} existence is shown when the generator $g(s,y,z)$ is strictly subquadratic in $z$ and satisfies some monotonicity condition. Multidimensional quadratic BSDEs appear in many applications, such as market making problems (see  Kramkov and Pulido \cite{KP}), nonzero-sum risk-sensitive stochastic differential games (see El Karoui and Hamad\`{e}ne \cite{EH}, Hu and Tang \cite{HT}) and non-zero sum differential games of BSDEs (see Hu and Tang \cite{HT1}).

Our results are motivated by the recent works of Hu and Tang \cite{HT} and Xing and \v{Z}itkovi\'{c} \cite{XZ}. We focus on the solvability of a type of BSDEs with triangularly quadratic generators. More precisely, we study the coupled system of quadratic BSDEs
\begin{align*}
\begin{cases}
&Y^{1}_{t}=\xi^1+\int_t^T\left[\frac{1}{2}|Z^1_s|^2+Z^1_sl^1(s,Y_s)+h^1(s,Y_s,Z_s)\right]ds-\int_t^TZ^{1}_sdW_s,\\
&Y^{i}_{t}=\xi^i+\int_t^T\left[\frac{1}{2}|Z^i_s|^2+Z^{i}_sl^i(s,Y_s,Z_s)-k^i(s,Z_s)+h^i(s,Y_s,Z_s)\right]ds\\
&\quad \quad \quad -\int_t^TZ^{i}_sdW_s,\quad,\quad i=2,\ldots,n.
\end{cases}
\end{align*}
By borrowing some techniques from Hu and Tang \cite{HT}, we first prove that this type of triangularly quadratic BSDEs admits a unique local solution on a small time interval whenever the terminal value is bounded using a contraction mapping argument. Under additional assumptions, we show that the value process is uniformly bounded. Therefore we build the global solution on the whole time interval by stitching local solutions. Finally, we give solvability results when the generators have path dependence in value process.

The paper is organized as follows. In Section 2, we state the setting and main results. Solvability results for a type of quadratic BSDEs with path dependence in value process are presented in Section 3.
\section{Preliminaries and main results}
Let $W=(W_t)_{t\geq 0}$ be a $d$-dimensional Brownian motion on a probability space $(\Omega, {\cal F}, P)$. Let $(\mathcal{F}_t)_{t\geq 0}$ be the augmented filtration generated by $W$. Throughout, we fix a $T\in (0,\infty)$. We endow $\Omega \times [0,T]$ with the predictable $\sigma$-algebra $\mathcal{P}$ and $\mathbb{R}^n$ with its Borel $\sigma$-algebra $\mathcal{B}(\mathbb{R}^n)$. Equalities and inequalities between random variables and processes are understood in the $P$-a.s. and $P\otimes dt$-a.e. sense, respectively. The Euclidean norm is denoted by $|\cdot|$ and $\|\cdot\|_\infty$ denotes the $L^\infty$-norm. $\mathcal{C}_{T}(\mathbb{R}^n)$ denotes the set $C([0,T];\mathbb{R}^n)$ of continuous functions from $[0,T]$ to $\mathbb{R}^n$. For $p>1$, we denote by
\begin{itemize}
\item $\mathcal{S}^p(\mathbb{R}^n)$ the set of $n$-dimensional continuous adapted processes $Y$ on $[0,T]$ such that
\begin{equation*}
\|Y\|_{\mathcal{S}^p}:=E\left[\sup_{0\leq t\leq T} |Y_t|^p\right]^{\frac{1}{p}}< \infty;
\end{equation*}
\item $\mathcal{S}^\infty(\mathbb{R}^n)$ the set of $n$-dimensional continuous adapted processes $Y$ on $[0,T]$ such that
\begin{equation*}
\|Y\|_{\mathcal{S}^\infty}:=\bigg\| \sup_{0\leq t\leq T} |Y_t| \bigg\|_{\infty} < \infty;
\end{equation*}
\item $\mathcal{H}^p(\mathbb{R}^{n\times d})$ the set of predictable $\mathbb{R}^{n\times d}$-valued processes $Z$ such that
    \begin{equation*}
    \|Z\|_{\mathcal{H}^p}=E\left[\left(\int_0^T|Z_s|^2ds\right)^{\frac{p}{2}}\right]^{\frac{1}{p}}<\infty.
    \end{equation*}
\end{itemize}
Let $\mathcal{T}$ be the set of all stopping times with values in $[0,T]$. For any uniformly integrable martingale $M$ with $M_0=0$ and for $p\geq 1$, we set
\begin{equation*}
\|M\|_{BMO_p(P)}:=\sup_{\tau\in\mathcal{T}}\|E[|M_T-M_\tau|^p|\mathcal{F}_\tau]^{\frac{1}{p}}\|_\infty.
\end{equation*}
The class $\{M:~\|M\|_{BMO_p}<\infty\}$ is denoted by $BMO_p$, which is written as $BMO(P)$ when it is necessary to indicate the underlying probability measure $P$. In particular, we will denote it by $BMO$ when $p=2$. For $(\alpha\cdot W)_t:=\int_0^t\alpha_sdW_s$ in $BMO$, the corresponding stochastic exponential is denoted by $\mathcal{E}_{t}(\alpha\cdot W)$. We recall classical results on $BMO$ spaces (see \cite[Theorem 3.6]{Ka} and \cite[Corollary 2.1]{Ka}).
\begin{lemma}\label{PP}
Let $a\cdot W\in BMO$ be such that $\|a\cdot W\|_{BMO}\leq \gamma$ for some $\gamma\geq 0$, and $\tilde{P}$ be given by $\frac{d\tilde{P}}{dP}:=\mathcal{E}_T(a\cdot W)$, under which $\tilde{W}=W-\int_0^{\cdot}a_sds$ is a Brownian motion. Then for every $b\cdot W\in BMO$, there exist two constants $\delta(\gamma)$ and $\Delta(\gamma)$ only depending on $\gamma$ such that
\begin{equation*}
 \delta(\gamma)\|b\cdot W\|^2_{BMO} \leq \| b\cdot \tilde{W}\|^2_{BMO(\tilde{P})} \leq \Delta(\gamma) \| b\cdot W\|^2_{BMO}.
 \end{equation*}
\end{lemma}
\begin{lemma}\label{EQ}
For any $p\geq 1$, there is a generic constant $L_p>0$ such that for any uniformly integrable martingale $M$,
\begin{equation*}
\|M\|^2_{BMO_p}\leq L_p\|M\|^2_{BMO_2}.
\end{equation*}
\end{lemma}
We consider the following function from $\mathbb{R}$ into itself defined by
\begin{equation*}
u(x)=e^{x}-1-x.
\end{equation*}
It is easy to check that $u$ has the following properties
\begin{equation*}
\begin{cases}
&u(x)\geq 0,\\
&u(x)\geq |x|-1,\\
&u''(x)-u'(x)=1.
\end{cases}
\end{equation*}
We will now focus on the solvability of the following type of BSDEs with triangularly quadratic generators:
\begin{align}\label{Eq1}
\begin{cases}
&Y^{1}_{t}=\xi^1+\int_t^T\left[\frac{1}{2}|Z^1_s|^2+Z^1_sl^1(s,Y_s)+h^1(s,Y_s,Z_s)\right]ds-\int_t^TZ^{1}_sdW_s,\\
&Y^{i}_{t}=\xi^i+\int_t^T\left[\frac{1}{2}|Z^i_s|^2+Z^{i}_sl^i(s,Y_s,Z_s)-k^i(s,Z_s)+h^i(s,Y_s,Z_s)\right]ds\\
&\quad \quad \quad -\int_t^TZ^{i}_sdW_s,\quad,\quad i=2,\ldots,n,
\end{cases}
\end{align}
where $\xi$ is $\mathbb{R}^d$-valued and $\mathcal{F}_T$-measurable random variable which is bounded. Let $C$ be a positive constant, we will make the following assumptions:
\begin{itemize}
\item[(A1)] $l^1:\Omega\times[0,T]\times\mathbb{R}^n\rightarrow\mathbb{R}$ satisfies that $l^1(\cdot,y)$ is adapted for each $y\in\mathbb{R}^n$. Moreover, it holds that
    \begin{align*}
    &|l^1(t,y)|\leq C(1+|y|),~~y\in\mathbb{R}^n;\\
    &|l^1(t,y)-l^1(t,\bar{y})|\leq C|y-\bar{y}|,~~~y,\bar{y}\in\mathbb{R}^n;
    \end{align*}
\item[(A2)]
    For $i=2,\ldots,n$, $l^i:\Omega\times[0,T]\times\mathbb{R}^n\times\mathbb{R}^{n\times d}\rightarrow\mathbb{R}$ satisfies that $l^i(\cdot,y,z)$ is adapted for each $y\in\mathbb{R}^n$ and $z\in\mathbb{R}^{n\times d}$. Moreover, it holds that
    \begin{align*}
    &|l^i(t,y,z)|\leq C(1+|y|),~~y\in\mathbb{R}^n,~z\in\mathbb{R}^{n\times d};\\
    &|l^i(t,y,z)-l^i(t,\bar{y},\bar{z})|\leq C|y-\bar{y}|+C\sum_{j=1}^{i-1}|z^j-\bar{z}^j|,~~~y,\bar{y}\in\mathbb{R}^n， ~z,\bar{z}\in\mathbb{R}^{n\times d};
    \end{align*}
\item[(A3)] For $i=2,\ldots,n$, $k^i:\Omega\times[0,T]\times\mathbb{R}^{n\times d}\rightarrow\mathbb{R}$ satisfies that $k^i(\cdot,z)$ is adapted for each $z\in\mathbb{R}^{n\times d}$. Moreover, it holds that
    \begin{align*}
    &0\leq k^i(t,z)\leq C(1+\sum_{j=1}^{i-1}|z^j|^2),~~z\in\mathbb{R}^{n\times d};\\
    &|k^i(t,z)-k^i(t,\bar{z})|\leq C\sum_{j=1}^{i-1}(1+|z^j|+|\bar{z}^j|)|z^j-\bar{z}^j|,~~z,\bar{z}\in\mathbb{R}^{n\times d};
    \end{align*}
\item[(A4)] $h:\Omega\times[0,T]\times\mathbb{R}^n\times\mathbb{R}^{n\times d}\rightarrow\mathbb{R}^n$ satisfies that $h(\cdot,y,z)$ is adapted for each $y\in\mathbb{R}^n$ and $z\in\mathbb{R}^{n\times d}$. Moreover, there exists $\alpha\in[-1,1)$ such that
    \begin{align*}
    &|h(t,y,z)|\leq C(1+|y|+|z|^{1+\alpha}),~~y\in\mathbb{R}^n,~z\in\mathbb{R}^{n\times d};\\
    &|h(t,y,z)-h(t,\bar{y},\bar{z})|\leq C|y-\bar{y}|+C\left(1+|z|^{\alpha^+}+|\bar{z}|^{\alpha^+}\right)|z-\bar{z}|,
    \end{align*}
    for $y,\bar{y}\in\mathbb{R}^n$ and $z,\bar{z}\in\mathbb{R}^{n\times d}$.
\end{itemize}
\begin{remark}
Assumptions (A2) and (A3) implies that $l^i(t,y,z)$ and $k^i(t,z)$ will only depend on the first $i-1$ components of $z$ for $i=2,\ldots,n$.
\end{remark}
Our first result is the following theorem which concerns local solutions.
\begin{theorem}\label{local}
Assume (A1)-(A4) hold, then there exist constants $T_{\eta}$, $C_1$ and $C_2$ only depending on $\alpha$, $C$ and $\|\xi\|_{\infty}$ such that for $T\leq T_{\eta}$, BSDE \eqref{Eq1} admits a unique solution $(Y,Z)$ such that $(Y,Z\cdot W)\in\mathcal{S}^{\infty}(\mathbb{R}^n)\times BMO$ with $\|Y\|_{\mathcal{S}^{\infty}}\leq C_1$ and $\|Z\cdot W\|_{BMO}\leq C_2$.
\end{theorem}
\begin{proof}
The proof is divided into several steps.

\emph{Step 1.} We first show that for $(y,z\cdot W)\in\mathcal{S}^{\infty}(\mathbb{R}^n)\times BMO$, the following BSDE
\begin{equation*}
Y_s^1=\xi^1+\int_t^T\left(\frac{1}{2}|Z^1_s|^2+Z^1_sl^1(s,y_s)+h^1(s,y_s,z_s)\right)ds-\int_t^TZ^{1}_sdW_s
\end{equation*}
admits a unique solution $(Y^1,Z^1)$ such that $(Y^1,Z^1\cdot W)\in\mathcal{S}^{\infty}(\mathbb{R})\times BMO$. Indeed, noting that $\xi^1$ is bounded, $(y,z\cdot W)\in\mathcal{S}^{\infty}(\mathbb{R}^n)\times BMO$, $|l^1(t,y)|\leq C(1+|y|)$ and $|h^1(t,y,z)|\leq C(1+|y|+|z|^{1+\alpha})$, using John-Nirenberg inequality \cite[Theorem 2.2]{Ka} and Young's inequality, it is easy to check that the following BSDE
\begin{equation*}
\hat{Y}^1_s=e^{\xi^1+\int_0^Th(s,y_s,z_s)ds}+\int_t^T\hat{Z}^1_sl^1(s,y_s)ds-\int_t^T\hat{Z}^1_sdW_s
\end{equation*}
admits a unique solution $(\hat{Y}^1,\hat{Z}^1)\in\mathcal{S}^p(\mathbb{R})\times\mathcal{H}^p(\mathbb{R}^d)$ for any $p>1$. Moreover, $\hat{Y}^1>0$. Denoting $\tilde{Y}^1_t=\ln \hat{Y}^1_t$, we have
\begin{equation*}
\tilde{Y}_t^1=\left(\xi^1+\int_0^Th(s,y_s,z_s)ds\right)+\int_t^T\left[\frac{1}{2}\bigg|\frac{\hat{Z}^1_s}{\hat{Y}^1_s}\bigg|^2+\frac{\hat{Z}^1_s}{\hat{Y}^1_s}l^1(s,y_s)\right]ds-\int_t^T\frac{\hat{Z}^1_s}{\hat{Y}^1_s}dW_s.
\end{equation*}
Let $Y^1_t=\tilde{Y}^1_t-\int_0^th(s,y_s,z_s)ds$ and $Z^1_t=\frac{\hat{Z}^1_t}{\hat{Y}^1_t}$, then $(Y^1,Z^1)$ satisfies
\begin{equation*}
Y_t^1=\xi^1+\int_t^T\left(\frac{1}{2}|Z^1_s|^2+Z^1_sl^1(s,y_s)+h^1(s,y_s,z_s)\right)ds-\int_t^TZ^{1}_sdW_s.
\end{equation*}
Now we will show that $(Y^1,Z^1\cdot W)\in\mathcal{S}^{\infty}(\mathbb{R})\times BMO$. Actually, we have
\begin{equation*}
\hat{Y}^1_t=E^{Q^1}\left[e^{\xi^1+\int_0^Th(s,y_s,z_s)ds}\big|\mathcal{F}_t\right],
\end{equation*}
where $Q^1$ is the equivalent probability measure given by $\frac{dQ^1}{dP}=\mathcal{E}_T(l^1(\cdot,y_{\cdot})\cdot W)$. Therefore, it holds that
\begin{align*}
Y^1_t=\ln E^{Q^1}\left[e^{\xi^1+\int_t^Th(s,y_s,z_s)ds}\big|\mathcal{F}_t\right].
\end{align*}
Hence we have
\begin{align*}
Y^1_t&\leq \ln E^{Q^1}\left[e^{|\xi^1|+\int_t^T|h(s,y_s,z_s)|ds}\big|\mathcal{F}_t\right]\\
&\leq \ln E^{Q^1}\left[e^{|\xi^1|+\int_t^TC(1+|y_s|+|z_s|^{1+\alpha})ds}\big|\mathcal{F}_t\right]\\
&\leq \|\xi^1\|_{\infty}+C(1+\|y\|_{\infty})(T-t)+\ln E^{Q^1}\left[e^{C\int_t^T|z_s|^{1+\alpha}ds}\big|\mathcal{F}_t\right].
\end{align*}
Using Young's inequality, we obtain that
\begin{align*}
&E^{Q^1}\left[e^{C\int_t^T|z_s|^{1+\alpha}ds}\big|\mathcal{F}_t\right]\\
&\leq E^{Q^1}\left[e^{\int_t^T\frac{1+\alpha}{2}\frac{|z_s|^{2}}{(1+\alpha)\|z\cdot W\|^2_{BMO(Q^1)}}+\frac{1-\alpha}{2}C^{\frac{1+\alpha}{1-\alpha}}\left((1+\alpha)\|z\cdot W\|^2_{BMO(Q^1)}\right)^{\frac{1+\alpha}{1-\alpha}}ds}\bigg|\mathcal{F}_t\right].
\end{align*}
Applying John-Nirenberg inequality \cite[Theorem 2.2]{Ka}, it holds that
\begin{equation*}
E^{Q^1}\left[e^{\int_t^T\frac{1+\alpha}{2}\frac{|z_s|^{2}}{(1+\alpha)\|z\cdot W\|^2_{BMO(Q^1)}}ds}\bigg|\mathcal{F}_t\right]\leq 2.
\end{equation*}
Therefore, we have
\begin{equation*}
Y^1_t\leq\|\xi^1\|_{\infty}+C(1+\|y\|_{\infty})(T-t)+\frac{1-\alpha}{2}C^{\frac{1+\alpha}{1-\alpha}}((1+\alpha)\|z\cdot W\|^2_{BMO(Q^1)})^{\frac{1+\alpha}{1-\alpha}}(T-t)+\ln 2.
\end{equation*}
On the other hand, it holds that
\begin{align*}
Y^1_t&=E^{Q^1}\left[\xi^1+\int_t^T\left(\frac{1}{2}|Z^1_s|^2+h^1(s,y_s,z_s)\right)ds\bigg|\mathcal{F}_t\right]\\
&\geq E^{Q^1}\left[\xi^1+\int_t^Th^1(s,y_s,z_s)ds\bigg|\mathcal{F}_t\right]\\
&\geq E^{Q^1}\left[-|\xi^1|-C\int_t^T(1+|y_s|+|z_s|^{1+\alpha})ds\bigg|\mathcal{F}_t\right]\\
&\geq -\|\xi^1\|_{\infty}-C(1+\|y\|_{\infty})(T-t)-CE^{Q^1}\left[\int_t^T|z_s|^{1+\alpha}ds\bigg|\mathcal{F}_t\right].
\end{align*}
Using Young's inequality, we have
\begin{align}\label{Es1}
E^{Q^1}\left[\int_t^T|z_s|^{1+\alpha}ds\bigg|\mathcal{F}_t\right]
&\leq \frac{1+\alpha}{2}\frac{1}{(1+\alpha)\|z\cdot W\|^2_{BMO(Q^1)}}E^{Q^1}\left[\int_t^T|z_s|^2ds\bigg|\mathcal{F}_t\right]\nonumber\\
&\quad\quad+\frac{1-\alpha}{2}\left((1+\alpha)\|z\cdot W\|^2_{BMO(Q^1)}\right)^{\frac{1+\alpha}{1-\alpha}}(T-t)\\
&\leq \frac{1}{2}+\frac{1-\alpha}{2}\left((1+\alpha)\|z\cdot W\|^2_{BMO(Q^1)}\right)^{\frac{1+\alpha}{1-\alpha}}(T-t)\nonumber.
\end{align}
Hence, it holds that
\begin{align*}
|Y^1_t|&\leq\|\xi^1\|_{\infty}+C(1+\|y\|_{\infty})(T-t)+\ln 2+\frac{1}{2}\\
&\quad +\frac{1-\alpha}{2}(C^{\frac{1+\alpha}{1-\alpha}}+1)\left((1+\alpha)\|z\cdot W\|^2_{BMO(Q^1)}\right)^{\frac{1+\alpha}{1-\alpha}}(T-t).
\end{align*}
Applying It\^{o}'s formula to $u(Y^1_t)$, we obtain that
\begin{align*}
&u(Y^{1}_t)=u(\xi^1)-\int_t^Tu'(Y^{1}_s)Z^{1}_sdW_s\\
&+\int_t^T\left(u'(Y^{1}_s)\left(\frac{1}{2}|Z^{1}_s|^2+Z^1_sl^1(s,y_s)+h^1(s,y_s,z_s)\right)-\frac{1}{2}u''(Y^{1}_s)|Z^{1}_s|^2\right)ds\\
&=u(\xi^1)-\int_t^Tu'(Y^{1}_s)Z^{1}_sdW_s-\frac{1}{2}\int_t^T|Z^{1}_s|^2ds\\
&+\int_t^Tu'(Y^{1}_s)\left(Z^1_sl^1(s,y_s)+h^1(s,y_s,z_s)\right)ds.
\end{align*}
Therefore, we have
\begin{align*}
&E^{Q^1}\left[\frac{1}{2}\int_t^T|Z^{1}_s|^2ds\bigg|\mathcal{F}_t\right]\\
&\leq E^{Q^1}\left[u(\xi^1)+\int_t^Tu'(Y^{1}_s)h^1(s,y_s,z_s)ds\bigg|\mathcal{F}_t\right]\\
&\leq e^{\|\xi^1\|_{\infty}}+\|\xi^1\|_{\infty}+Ce^{\|\xi^1\|_{\infty}}E^{Q^1}\left[\int_t^T(1+|y_s|+|z_s|^{1+\alpha})ds\bigg|\mathcal{F}_t\right]\\
&\leq e^{\|\xi^1\|_{\infty}}+\|\xi^1\|_{\infty}+Ce^{\|\xi^1\|_{\infty}}(1+\|y\|_{\infty})(T-t)+Ce^{\|\xi^1\|_{\infty}}E^{Q^1}\left[\int_t^T|z_s|^{1+\alpha}ds\bigg|\mathcal{F}_t\right].
\end{align*}
Noting the inequality \eqref{Es1}, it holds that
\begin{align*}
E^{Q^1}\left[\frac{1}{2}\int_t^T|Z^{1}_s|^2ds|\mathcal{F}_t\right]&\leq \left(1+\frac{C}{2}\right)e^{\|\xi^1\|_{\infty}}+\|\xi^1\|_{\infty}+Ce^{\|\xi^1\|_{\infty}}(1+\|y\|_{\infty})(T-t)\\
&\quad +\frac{1-\alpha}{2}Ce^{\|\xi^1\|_{\infty}}((1+\alpha)\|z\cdot W\|^2_{BMO(Q^1)})^{\frac{1+\alpha}{1-\alpha}}(T-t).
\end{align*}
Therefore by denoting $\Delta:=\Delta(C(1+\|y\|_{\mathcal{S}^{\infty}})\sqrt{T})$ and  $\delta:=\delta(C(1+\|y\|_{\mathcal{S}^{\infty}})\sqrt{T})$ , we obtain that
\begin{equation*}
\|Y^1\|_{\mathcal{S}^{\infty}}\leq\|\xi^1\|_{\infty}+C(1+\|y\|_{\infty})T+\ln 2+\frac{1}{2}+\frac{1-\alpha}{2}(C^{\frac{1+\alpha}{1-\alpha}}+1)((1+\alpha)\Delta\|z\cdot W\|^2_{BMO})^{\frac{1+\alpha}{1-\alpha}}T
\end{equation*}
and
\begin{align*}
\|Z^{1}\cdot W\|^2_{BMO}&\leq \frac{2}{\delta}\left(\left(1+\frac{C}{2}\right)e^{\|\xi^1\|_{\infty}}+\|\xi^1\|_{\infty}+Ce^{\|\xi^1\|_{\infty}}(1+\|y\|_{\infty})T\right.\\
&\quad \left.+\frac{1-\alpha}{2}Ce^{\|\xi^1\|_{\infty}}((1+\alpha)\Delta\|z\cdot W\|^2_{BMO})^{\frac{1+\alpha}{1-\alpha}}T\right).
\end{align*}
\emph{Step 2.} Similar to the first step, it is easy to check that
\begin{equation*}
Y^{2}_{t}=\xi^2+\int_t^T\left(\frac{1}{2}|Z^2_s|^2+Z^{2}_sl^2(s,y_s,Z^1_s)-k^2(s,Z^1_s)+h^2(s,y_s,z_s)\right)dt-\int_t^TZ^{2}_sdW_s
\end{equation*}
admits a unique solution $(Y^2,Z^2)$ such that $(Y^2,Z^2\cdot W)\in\mathcal{S}^{\infty}(\mathbb{R})\times BMO$. Moreover, we have
\begin{align*}
|Y^2_t|&\leq\|\xi^2\|_{\infty}+C(2+\|y\|_{\infty})(T-t)+C\|Z^1\cdot W\|^2_{BMO(Q^2)}+\ln 2+\frac{1}{2}\\
&\quad +\frac{1-\alpha}{2}(C^{\frac{1+\alpha}{1-\alpha}}+1)((1+\alpha)\|z\cdot W\|^2_{BMO(Q^2)})^{\frac{1+\alpha}{1-\alpha}}(T-t)
\end{align*}
and
\begin{align*}
&E^{Q^2}\left[\frac{1}{2}\int_t^T|Z^{2}_s|^2ds\bigg|\mathcal{F}_t\right]\\
&\leq \left(1+\frac{C}{2}\right)e^{\|\xi^2\|_{\infty}}+\|\xi^2\|_{\infty}+Ce^{\|\xi^2\|_{\infty}}(2+\|y\|_{\infty})(T-t)\\
&\quad+\frac{1-\alpha}{2}Ce^{\|\xi^2\|_{\infty}}((1+\alpha)\|z\cdot W\|^2_{BMO(Q^2)})^{\frac{1+\alpha}{1-\alpha}}(T-t)+C\|Z^1\cdot W\|^2_{BMO(Q^2)}
\end{align*}
where $Q^2$ is the equivalent probability measure given by $\frac{dQ^2}{dP}=\mathcal{E}_T(l^2(\cdot,y_{\cdot},Z^1_{\cdot})\cdot W)$.
Therefore we obtain that
\begin{align*}
\|Y^2\|_{\mathcal{S}^{\infty}}&\leq\|\xi^2\|_{\infty}+C(2+\|y\|_{\infty})T+C\Delta\|Z^1\cdot W\|^2_{BMO}+\ln 2+\frac{1}{2}\\
&\quad +\frac{1-\alpha}{2}(C^{\frac{1+\alpha}{1-\alpha}}+1)((1+\alpha)\Delta\|z\cdot W\|^2_{BMO})^{\frac{1+\alpha}{1-\alpha}}T
\end{align*}
and
\begin{align*}
\|Z^{2}\cdot W\|^2_{BMO}&\leq\frac{2}{\delta}\left( \left(1+\frac{C}{2}\right)e^{\|\xi^2\|_{\infty}}+\|\xi^2\|_{\infty}+Ce^{\|\xi^2\|_{\infty}}(2+\|y\|_{\infty})T\right.\\
&\quad \left.+\frac{1-\alpha}{2}Ce^{\|\xi^2\|_{\infty}}((1+\alpha)\Delta\|z\cdot W\|^2_{BMO})^{\frac{1+\alpha}{1-\alpha}}T+C\Delta\|Z^1\cdot W\|^2_{BMO}\right).
\end{align*}
Recursively, we have that for $i=3,\ldots,n$, the following BSDE
\begin{equation*}
Y^{i}_{t}=\xi^i+\int_t^T\left(\frac{1}{2}|Z^i_s|^2+Z^{i}_sl^i(s,y_s,Z_s)-k^i(s,Z_s)+h^i(s,y_s,z_s)\right)dt-\int_t^TZ^{i}_sdW_s
\end{equation*}
admits a unique solution $(Y^i,Z^i)$ such that $(Y^i,Z^i\cdot W)\in\mathcal{S}^{\infty}(\mathbb{R})\times BMO$. Moreover, we have
\begin{align*}
|Y^i_t|&\leq\|\xi^i\|_{\infty}+C(2+\|y\|_{\infty})(T-t)+C\sum_{j=1}^{i-1}\|Z^j\cdot W\|^2_{BMO(Q^i)}+\ln 2+\frac{1}{2}\\
&\quad +\frac{1-\alpha}{2}(C^{\frac{1+\alpha}{1-\alpha}}+1)((1+\alpha)\|z\cdot W\|^2_{BMO(Q^i)})^{\frac{1+\alpha}{1-\alpha}}(T-t)
\end{align*}
and
\begin{align*}
&E^{Q^i}\left[\frac{1}{2}\int_t^T|Z^{i}_s|^2ds|\mathcal{F}_t\right]\\
&\leq \left(1+\frac{C}{2}\right)e^{\|\xi^i\|_{\infty}}+\|\xi^i\|_{\infty}+Ce^{\|\xi^i\|_{\infty}}(2+\|y\|_{\infty})(T-t)\\
&\quad+\frac{1-\alpha}{2}Ce^{\|\xi^i\|_{\infty}}((1+\alpha)\|z\cdot W\|^2_{BMO(Q^i)})^{\frac{1+\alpha}{1-\alpha}}(T-t)+C\sum_{j=1}^{i-1}\|Z^j\cdot W\|^2_{BMO(Q^i)}
\end{align*}
where $Q^i$ is the equivalent probability measure given by $\frac{dQ^i}{dP}=\mathcal{E}_T(l^i(\cdot,y_{\cdot},Z_{\cdot})\cdot W)$. Therefore we obtain that
\begin{align*}
\|Y^i\|_{\mathcal{S}^{\infty}}&\leq\|\xi^i\|_{\infty}+C(2+\|y\|_{\infty})T+C\Delta\sum_{j=1}^{i-1}\|Z^j\cdot W\|^2_{BMO}+\ln 2+\frac{1}{2}\\
&\quad +\frac{1-\alpha}{2}(C^{\frac{1+\alpha}{1-\alpha}}+1)((1+\alpha)\Delta\|z\cdot W\|^2_{BMO})^{\frac{1+\alpha}{1-\alpha}}T
\end{align*}
and
\begin{align*}
\|Z^{i}\cdot W\|^2_{BMO}&\leq\frac{2}{\delta}\left(\left(1+\frac{C}{2}\right)e^{\|\xi^i\|_{\infty}}+\|\xi^i\|_{\infty}+Ce^{\|\xi^i\|_{\infty}}(2+\|y\|_{\infty})T\right.\\
&\quad\left.+\frac{1-\alpha}{2}Ce^{\|\xi^i\|_{\infty}}((1+\alpha)\Delta_i\|z\cdot W\|^2_{BMO})^{\frac{1+\alpha}{1-\alpha}}T+C\Delta\sum_{j=1}^{i-1}\|Z^j\cdot W\|^2_{BMO}\right).
\end{align*}
\emph{Step 3.}
We will denote
\begin{align*}
&\Delta^*=\Delta(C),\\
&\delta^*=\delta(C),\\
&A=n\left(\|\xi\|_{\infty}+2+C+\frac{C^{\frac{1+\alpha}{1-\alpha}}}{2}+\frac{2n^3C\Delta^*}{\delta^*}\left(1+\left(\frac{2C\Delta^*}{\delta^*}\right)^{n}\right)\left(\left(1+2C\right)e^{\|\xi\|_{\infty}}+\|\xi\|_{\infty}\right)\right),\\
&B=\sqrt{\frac{2n^3}{\delta^*}\left(1+\left(\frac{2C\Delta^*}{\delta^*}\right)^{n}\right)\left(\left(1+2C\right)e^{\|\xi\|_{\infty}}+\|\xi\|_{\infty}\right)},\\
&\eta=\frac{1}{(2+A)^2}\wedge\frac{1}{(1-\alpha)((1+\alpha)\Delta^*B^2)^{\frac{1+\alpha}{1-\alpha}}}.
\end{align*}
Assuming that $T\leq \eta$, for $(y,z\cdot W)\in\mathcal{S}^{\infty}(\mathbb{R}^n)\times BMO$ such that $\|y\|_{\mathcal{S}^\infty}\leq A$ and $\|z\cdot W\|_{BMO}\leq B$, it follows from \emph{Step 1} and \emph{Step 2} that the following BSDE
\begin{align*}
\begin{cases}
&Y^{1}_{t}=\xi^1+\int_t^T\left[\frac{1}{2}|Z^1_s|^2+Z^1_sl^1(s,y_s)+h^1(s,y_s,z_s)\right]ds-\int_t^TZ^{1}_sdW_s,\\
&Y^{i}_{t}=\xi^i+\int_t^T\left[\frac{1}{2}|Z^i_s|^2+Z^{i}_sl^i(s,y_s,Z_s)-k^i(s,Z_s)+h^i(s,y_s,z_s)\right]dt\\
&\quad \quad \quad -\int_t^TZ^{i}_tdW_t,\quad,\quad i=2,\ldots,n.
\end{cases}
\end{align*}
admits a unique solution $(Y,Z)$ such that $(Y,Z\cdot W)\in\mathcal{S}^{\infty}(\mathbb{R}^n)\times BMO$. Moreover from Lemma \ref{PP}, $\Delta$ and $\delta$ can be replaced by $\Delta^*$ and $\delta^*$ respectively. Therefore, it holds that
\begin{align*}
\|Y^1\|_{\mathcal{S}^{\infty}}&\leq\|\xi^1\|_{\infty}+C(1+\|y\|_{\infty})T+\ln 2+\frac{1}{2}+\frac{1-\alpha}{2}(C^{\frac{1+\alpha}{1-\alpha}}+1)((1+\alpha)\Delta^*\|z\cdot W\|^2_{BMO})^{\frac{1+\alpha}{1-\alpha}}T\\
&\leq \|\xi^1\|_{\infty}+C(1+A)T+\ln 2+\frac{1}{2}+\frac{1-\alpha}{2}(C^{\frac{1+\alpha}{1-\alpha}}+1)((1+\alpha)\Delta^*B^2)^{\frac{1+\alpha}{1-\alpha}}T\\
&\leq \|\xi^1\|_{\infty}+2+C+\frac{C^{\frac{1+\alpha}{1-\alpha}}}{2}
\end{align*}
and
\begin{align*}
\|Z^{1}\cdot W\|^2_{BMO}&\leq \frac{2}{\delta^*}\left(\left(1+\frac{C}{2}\right)e^{\|\xi^1\|_{\infty}}+\|\xi^1\|_{\infty}+Ce^{\|\xi^1\|_{\infty}}(1+\|y\|_{\infty})T\right.\\
&\quad \left.+\frac{1-\alpha}{2}Ce^{\|\xi^1\|_{\infty}}((1+\alpha)\Delta^*\|z\cdot W\|^2_{BMO})^{\frac{1+\alpha}{1-\alpha}}T\right)\\
&\leq \frac{2}{\delta^*}\left(\left(1+\frac{C}{2}\right)e^{\|\xi^1\|_{\infty}}+\|\xi^1\|_{\infty}+Ce^{\|\xi^1\|_{\infty}}(1+A)T\right.\\
&\quad \left.+\frac{1-\alpha}{2}Ce^{\|\xi^1\|_{\infty}}((1+\alpha)\Delta^*B^2)^{\frac{1+\alpha}{1-\alpha}}T\right)\\
&\leq \frac{2}{\delta^*}\left(\left(1+2C\right)e^{\|\xi^1\|_{\infty}}+\|\xi^1\|_{\infty}\right).
\end{align*}
Similarly, we have
\begin{align*}
\|Y^i\|_{\mathcal{S}^{\infty}}&\leq\|\xi^i\|_{\infty}+C(2+\|y\|_{\infty})T+C\Delta^*\sum_{j=1}^{i-1}\|Z^j\cdot W\|^2_{BMO}+\ln 2+\frac{1}{2}\\
&\quad +\frac{1-\alpha}{2}(C^{\frac{1+\alpha}{1-\alpha}}+1)((1+\alpha)\Delta^*\|z\cdot W\|^2_{BMO})^{\frac{1+\alpha}{1-\alpha}}T\\
&\leq\|\xi^i\|_{\infty}+C(2+A)T+C\Delta^*\sum_{j=1}^{i-1}\|Z^j\cdot W\|^2_{BMO}+\ln 2+\frac{1}{2}\\
&\quad +\frac{1-\alpha}{2}(C^{\frac{1+\alpha}{1-\alpha}}+1)((1+\alpha)\Delta^*B^2)^{\frac{1+\alpha}{1-\alpha}}T\\
&\leq\|\xi^i\|_{\infty}+2+C+\frac{C^{\frac{1+\alpha}{1-\alpha}}}{2}+C\Delta^*\sum_{j=1}^{i-1}\|Z^j\cdot W\|^2_{BMO}\\
\end{align*}
and
\begin{align*}
\|Z^{i}\cdot W\|^2_{BMO}&\leq\frac{2}{\delta^*}\left(\left(1+\frac{C}{2}\right)e^{\|\xi^i\|_{\infty}}+\|\xi^i\|_{\infty}+Ce^{\|\xi^i\|_{\infty}}(2+\|y\|_{\infty})T\right.\\
&\quad\left.+\frac{1-\alpha}{2}Ce^{\|\xi^i\|_{\infty}}((1+\alpha)\Delta_i\|z\cdot W\|^2_{BMO})^{\frac{1+\alpha}{1-\alpha}}T+C\Delta^*\sum_{j=1}^{i-1}\|Z^j\cdot W\|^2_{BMO}\right)\\
&\leq\frac{2}{\delta^*}\left(\left(1+\frac{C}{2}\right)e^{\|\xi^i\|_{\infty}}+\|\xi^i\|_{\infty}+Ce^{\|\xi^i\|_{\infty}}(2+A)T\right.\\
&\quad\left.+\frac{1-\alpha}{2}Ce^{\|\xi^i\|_{\infty}}((1+\alpha)\Delta_iB^2)^{\frac{1+\alpha}{1-\alpha}}T+C\Delta^*\sum_{j=1}^{i-1}\|Z^j\cdot W\|^2_{BMO}\right)\\
&\leq\frac{2}{\delta^*}\left(\left(1+2C\right)e^{\|\xi^i\|_{\infty}}+\|\xi^i\|_{\infty}+C\Delta^*\sum_{j=1}^{i-1}\|Z^j\cdot W\|^2_{BMO}\right)
\end{align*}
Thus, we obtain recursively
\begin{align*}
\|Z^{i}\cdot W\|^2_{BMO}&\leq\frac{2i}{\delta^*}\sum_{j=1}^{i}\left(\frac{2C\Delta^*}{\delta^*}\right)^{i-j}\left(\left(1+2C\right)e^{\|\xi^i\|_{\infty}}+\|\xi^i\|_{\infty}\right)\\
&\leq\frac{2i}{\delta^*}\sum_{j=1}^{i}\left(\frac{2C\Delta^*}{\delta^*}\right)^{i-j}\left(\left(1+2C\right)e^{\|\xi\|_{\infty}}+\|\xi\|_{\infty}\right)\\
&\leq\frac{2n^2}{\delta^*}\left(1+\left(\frac{2C\Delta^*}{\delta^*}\right)^{n}\right)\left(\left(1+2C\right)e^{\|\xi\|_{\infty}}+\|\xi\|_{\infty}\right)
\end{align*}
and
\begin{align*}
\|Y^i\|_{\mathcal{S}^{\infty}}
&\leq\|\xi^i\|_{\infty}+2+C+\frac{C^{\frac{1+\alpha}{1-\alpha}}}{2}+C\Delta^*\sum_{j=1}^{i-1}\|Z^j\cdot W\|^2_{BMO}\\
&\leq \|\xi\|_{\infty}+2+C+\frac{C^{\frac{1+\alpha}{1-\alpha}}}{2}+\frac{2n^3C\Delta^*}{\delta^*}\left(1+\left(\frac{2C\Delta^*}{\delta^*}\right)^{n}\right)\left(\left(1+2C\right)e^{\|\xi\|_{\infty}}+\|\xi\|_{\infty}\right)
\end{align*}
Thus, it holds that $\|Y\|_{\mathcal{S}^{\infty}}\leq A$ and $\|Z\cdot W\|_{BMO}\leq B$.\\
\emph{Step 4.} We will denote
\begin{align*}
&\bar{\Delta}:=\Delta(\sqrt{2}C+2\sqrt{2}B)\\
&\bar{\delta}:=\delta(\sqrt{2}C+2\sqrt{2}B)\\
&\bar{A}=12C^2T(1+\bar{\Delta}B^2)\\
&\bar{B}=18C^2L^2_4\bar{\Delta}^2T^{1-\alpha^+}\left(3+2\alpha^{+} L_4\bar{\Delta}B^2\right)\\
&\bar{C}=6nC^2\bar{\Delta}^2\left(2B^2+3L^2_4\left(3+2L_4\bar{\Delta}B^2\right)\right)\\
&\bar{\eta}_1=\frac{1}{24C^2(1+\bar{\Delta}B^2)\left(n+\frac{\bar{C}n^4+n^3}{\bar{\delta}}\left(1+\left(\frac{\bar{C}}{\bar{\delta}}\right)^{n}\right)\right)}\\
&\bar{\eta}_2=\left(\frac{1}{36C^2L^2_4\bar{\Delta}^2\left(3+2\alpha^{+} L_4\bar{\Delta}B^2\right)\left(n+\frac{\bar{C}n^4+n^3}{\bar{\delta}}\left(1+\left(\frac{\bar{C}}{\bar{\delta}}\right)^{n}\right)\right)}\right)^{\frac{1}{1-\alpha^+}}.
\end{align*}
Assuming that $T\leq \eta\wedge\bar{\eta}_1\wedge\bar{\eta}_2$, for $(y,z\cdot W),(\bar{y},\bar{z}\cdot W)\in\mathcal{S}^{\infty}(\mathbb{R}^n)\times BMO$ such that $\|y\|_{\mathcal{S}^\infty}\leq A$, $\|z\cdot W\|_{BMO}\leq B$ and $\|\bar{y}\|_{\mathcal{S}^\infty}\leq A$, $\|\bar{z}\cdot W\|_{BMO}\leq B$, it follows from \emph{Step 1}, \emph{Step 2} and \emph{Step 3} that the following BSDE
\begin{align*}
\begin{cases}
&Y^{1}_{t}=\xi^1+\int_t^T\left(\frac{1}{2}|Z^1_s|^2+Z^1_sl^1(s,y_s)+h^1(s,y_s,z_s)\right)ds-\int_t^TZ^{1}_sdW_s,\\
&Y^{i}_{t}=\xi^i+\int_t^T\left(\frac{1}{2}|Z^i_s|^2+Z^{i}_sl^i(s,y_s,Z_s)-k^i(s,Z_s)+h^i(s,y_s,z_s)\right)ds\\
&\quad \quad \quad -\int_t^TZ^{i}_sdW_s,\quad,\quad i=2,\ldots,n.
\end{cases}
\end{align*}
and
\begin{align*}
\begin{cases}
&\bar{Y}^{1}_{t}=\xi^1+\int_t^T\left(\frac{1}{2}|\bar{Z}^1_s|^2+\bar{Z}^1_sl^1(s,\bar{y}_s)+h^1(s,\bar{y}_s,\bar{z}_s)\right)ds-\int_t^T\bar{Z}^{1}_sdW_s,\\
&\bar{Y}^{i}_{t}=\xi^i+\int_t^T\left(\frac{1}{2}|\bar{Z}^i_s|^2+\bar{Z}^{i}_sl^i(s,\bar{y}_s,\bar{Z}_s)-k^i(s,\bar{Z}_s)+h^i(s,\bar{y}_s,\bar{z}_s)\right)ds\\
&\quad \quad \quad -\int_t^TZ^{i}_sdW_s,\quad,\quad i=2,\ldots,n.
\end{cases}
\end{align*}
admit unique solutions $(Y,Z)$ and $(\bar{Y},\bar{Z})$ respectively such that $(Y,Z\cdot W), (\bar{Y},\bar{Z}\cdot W)\in\mathcal{S}^{\infty}(\mathbb{R}^n)\times BMO$ with $\|Y\|_{\mathcal{S}^\infty}\leq A$, $\|Z\cdot W\|_{BMO}\leq B$ and $\|\bar{Y}\|_{\mathcal{S}^\infty}\leq A$, $\|\bar{Z}\cdot W\|_{BMO}\leq B$. Then, we have
\begin{align*}
Y^1_t-\bar{Y}^1_t&=\int_{t}^T\left(\frac{1}{2}|Z^1_s|^2-\frac{1}{2}|\bar{Z}^1_s|^2+Z^1_sl^1(s,y_s)-\bar{Z}^1_sl^1(s,\bar{y}_s)+h^1(s,y_s,z_s)-h^1(s,\bar{y}_s,\bar{z}_s)\right)ds\\
&\quad\quad-\int_t^TZ^1_s-\bar{Z}^1_sdW_s\\
&=\int_{t}^T\bar{Z}^1_s\left(l^1(s,y_s)-l^1(s,\bar{y}_s)\right)+h^1(s,y_s,z_s)-h^1(s,\bar{y}_s,\bar{z}_s)ds-\int_t^TZ^1_s-\bar{Z}^1_sd\bar{W}^1_s,
\end{align*}
where
\begin{equation*}
\bar{W}^1_t=W_t-\int_{0}^t\left(l^1(s,y_s)+\frac{1}{2}(Z^1_s+\bar{Z}^1_s)\right)ds
\end{equation*}
is a Brownian motion under an equivalent probability measure $\bar{P}^1$ defined by
\begin{equation*}
\frac{d\bar{P}^1}{dP}=\mathcal{E}_{T}\left(\left(l^1(\cdot,y_{\cdot})+\frac{1}{2}(Z^1_{\cdot}+\bar{Z}^1_{\cdot})\right)\cdot W\right).
\end{equation*}
For any stopping time $\tau$ taking values in $[0,T]$, we have
\begin{align*}
&|Y^1_{\tau}-\bar{Y}^1_{\tau}|^2+\bar{E}^1\left[\int_{\tau}^T|Z^1_s-\bar{Z}^1_s|^2ds\bigg|\mathcal{F}_{\tau}\right]\\
&=\bar{E}^1\left[\left(\int_{\tau}^T\bar{Z}^1_s\left(l^1(s,y_s)-l^1(s,\bar{y}_s)\right)+h^1(s,y_s,z_s)-h^1(s,\bar{y}_s,\bar{z}_s)ds\right)^2\bigg|\mathcal{F}_{\tau}\right]\\
&\leq C^2\bar{E}^1\left[\left(\int_{\tau}^T\left(|\bar{Z}^1_s||y_s-\bar{y}_s|+|y_s-\bar{y}_s|+(1+|z_s|^{\alpha^+}+|\bar{z}_s|^{\alpha^+})|z_s-\bar{z}_s|\right)ds\right)^2\bigg|\mathcal{F}_{\tau}\right]\\
&\leq 3C^2(T-\tau)^2\|y-\bar{y}\|^2_{\mathcal{S}^{\infty}}+3C^2(T-\tau)\|y-\bar{y}\|^2_{\mathcal{S}^{\infty}}\bar{E}^1\left[\int_{\tau}^T|\bar{Z}^1_s|^2ds\bigg|\mathcal{F}_{\tau}\right]\\
&\quad+3C^2\bar{E}^1\left[\int_{\tau}^T(1+|z_s|^{\alpha^+}+|\bar{z}_s|^{\alpha^+})^2ds\int_{\tau}^T|z_s-\bar{z}_s|^2ds\bigg|\mathcal{F}_{\tau}\right]\\
&\leq 3C^2(T-\tau)^2\|y-\bar{y}\|^2_{\mathcal{S}^{\infty}}+3C^2(T-\tau)\|y-\bar{y}\|^2_{\mathcal{S}^{\infty}}\bar{E}^1\left[\int_{\tau}^T|\bar{Z}^1_s|^2ds\bigg|\mathcal{F}_{\tau}\right]\\
&\quad+9C^2\bar{E}^1\left[\left(\int_{\tau}^T(1+|z_s|^{2\alpha^+}+|\bar{z}_s|^{2\alpha^+})ds\right)^2\bigg|\mathcal{F}_{\tau}\right]^{\frac{1}{2}}\bar{E}^1\left[\left(\int_{\tau}^T|z_s-\bar{z}_s|^2ds\right)^2\bigg|\mathcal{F}_{\tau}\right]^{\frac{1}{2}}.
\end{align*}
It follows from Lemma \ref{PP} and Lemma \ref{EQ} that
\begin{equation*}
\bar{E}^1\left[\int_{\tau}^T|\bar{Z}^1_s|^2ds\bigg|\mathcal{F}_{\tau}\right]\leq \|\bar{Z}^1\cdot\bar{W}^1\|^2_{BMO(\bar{P}^1)}\leq \bar{\Delta}\|\bar{Z}^1\cdot W\|^2_{BMO}
\end{equation*}
and
\begin{align*}
\bar{E}^1\left[\left(\int_{\tau}^{T}|z_s-\bar{z}_s|^2ds\right)^2\bigg|\mathcal{F}_{\tau}\right]^{\frac{1}{2}}\leq L_4\|(z-\bar{z})\cdot \bar{W}^1\|^2_{BMO(\bar{P}^1)}\leq L_4\bar{\Delta}\|(z-\bar{z})\cdot W\|^2_{BMO}
\end{align*}
and
\begin{align*}
&\bar{E}^1\left[\left(\int_{\tau}^T(1+|z_s|^{2\alpha^+}+|\bar{z}_s|^{2\alpha^+})ds\right)^2\bigg|\mathcal{F}_{\tau}\right]^{\frac{1}{2}}\\
&\leq\bar{E}^1\left[\left(T+T^{1-\alpha^{+}}\left(\int_{\tau}^T|z_s|^2ds\right)^{\alpha^{+}}+T^{1-\alpha^{+}}\left(\int_{\tau}^T|\bar{z}_s|^2ds\right)^{\alpha^{+}}\right)^2\bigg|\mathcal{F}_{\tau}\right]^{\frac{1}{2}}\\
&\leq T^{1-\alpha^{+}}\bar{E}^1\left[\left(T^{\alpha^+}+\left(\int_{\tau}^T|z_s|^2ds\right)^{\alpha^+}+\left(\int_{\tau}^T|\bar{z}_s|^2ds\right)^{\alpha^+}\right)^2\bigg|\mathcal{F}_{\tau}\right]^{\frac{1}{2}}\\
&\leq T^{1-\alpha^+}\bar{E}^1\left[\left(T^{\alpha^+}+2-2\alpha^{+}+\alpha^{+}\int_{\tau}^T|z_s|^2ds+\alpha^{+}\int_{\tau}^T|\bar{z}_s|^2ds\right)^2\bigg|\mathcal{F}_{\tau}\right]^{\frac{1}{2}}\\
&\leq
T^{1-\alpha^+}\left(T^{\alpha^+}+2+\alpha^+\bar{E}^1\left[\left(\int_{\tau}^T|z_s|^2ds\right)^2\bigg|\mathcal{F}_{\tau}\right]^{\frac{1}{2}}+\alpha^+\bar{E}^1\left[\left(\int_{\tau}^T|\bar{z}_s|^2ds\right)^2\bigg|\mathcal{F}_{\tau}\right]^{\frac{1}{2}}\right)\\
&\leq
T^{1-\alpha^+}\left(T^{\alpha^+}+2+\alpha^+ L_4\|z\cdot W\|^2_{BMO(\bar{P}^1)}+\alpha^+ L_4\|\bar{z}\cdot W\|^2_{BMO(\bar{P}^1)}\right)\\
&\leq
T^{1-\alpha^+}\left(T^{\alpha^+}+2+\alpha^{+} L_4\bar{\Delta}\|z\cdot W\|^2_{BMO}+\alpha^{+} L_4\bar{\Delta}\|\bar{z}\cdot W\|^2_{BMO}\right).
\end{align*}
Therefore, we have
\begin{align*}
&|Y^1_{\tau}-\bar{Y}^1_{\tau}|^2+\bar{E}^1\left[\int_{\tau}^T|Z^1_s-\bar{Z}^1_s|^2|\mathcal{F}_{\tau}\right]\\
&\leq 3C^2T(T+\bar{\Delta}B^2)\|y-\bar{y}\|^2_{\mathcal{S}^{\infty}}+9C^2L^2_4\bar{\Delta}^2T^{1-\alpha^+}\left(T^{\alpha^+}+2+2\alpha^{+} L_4\bar{\Delta}B^2\right)\|(z-\bar{z})\cdot W\|^2_{BMO}.
\end{align*}
Hence
\begin{align*}
&\|Y^1-\bar{Y}^1\|^2_{\mathcal{S}^{\infty}}+\bar{\delta}\|Z^1_s-\bar{Z}^1_s\|^2_{BMO}\\
&\leq 6C^2T(T+\bar{\Delta}B^2)\|y-\bar{y}\|^2_{\mathcal{S}^{\infty}}+18C^2L^2_4\bar{\Delta}^2T^{1-\alpha^+}\left(T^{\alpha^+}+2+2\alpha^{+} L_4\bar{\Delta}B^2\right)\|(z-\bar{z})\cdot W\|^2_{BMO},\\
&\leq \bar{A}\|y-\bar{y}\|^2_{\mathcal{S}^{\infty}}+\bar{B}\|(z-\bar{z})\cdot W\|^2_{BMO}.
\end{align*}
Similarly, we could get that
\begin{align*}
&\|Y^i-\bar{Y}^i\|^2_{\mathcal{S}^{\infty}}+\bar{\delta}\|Z^i_s-\bar{Z}^i_s\|^2_{BMO}\\
&\leq 12C^2T(T+\bar{\Delta}B^2)\|y-\bar{y}\|^2_{\mathcal{S}^{\infty}}+12iC^2\bar{\Delta}^2B^2\sum_{j=1}^{i-1}\|(Z^j-\bar{Z}^j)\cdot W\|^2_{BMO}\\
&\quad +18iC^2L^2_4\bar{\Delta}^2\left(T+2+2L_4\bar{\Delta}B^2\right)\sum_{j=1}^{i-1}\|(Z^j-\bar{Z}^j)\cdot W\|^2_{BMO}\\
&\quad +18C^2L^2_4\bar{\Delta}^2T^{1-\alpha^+}\left(T^{\alpha^+}+2+2\alpha^{+} L_4\bar{\Delta}B^2\right)\|(z-\bar{z})\cdot W\|^2_{BMO}\\
&\leq \bar{A}\|y-\bar{y}\|^2_{\mathcal{S}^{\infty}}+\bar{C}\sum_{j=1}^{i-1}\|(Z^j-\bar{Z}^j)\cdot W\|^2_{BMO}+\bar{B}\|(z-\bar{z})\cdot W\|^2_{BMO}.
\end{align*}
Thus, we obtain recursively
\begin{align*}
\|(Z^i_s-\bar{Z}^i_s)\cdot W\|^2_{BMO}&\leq\frac{i}{\bar{\delta}}\sum_{j=1}^{i}\left(\frac{\bar{C}}{\bar{\delta}}\right)^{i-j}\left(\bar{A}\|y-\bar{y}\|^2_{\mathcal{S}^{\infty}}+\bar{B}\|(z-\bar{z})\cdot W\|^2_{BMO}\right)\\
&\leq\frac{n^2}{\bar{\delta}}\left(1+\left(\frac{\bar{C}}{\bar{\delta}}\right)^{n}\right)\left(\bar{A}\|y-\bar{y}\|^2_{\mathcal{S}^{\infty}}+\bar{B}\|(z-\bar{z})\cdot W\|^2_{BMO}\right)
\end{align*}
and
\begin{align*}
\|Y^i-\bar{Y}^i\|^2_{\mathcal{S}^{\infty}}&\leq\left(1+\frac{\bar{C}n^3}{\bar{\delta}}\left(1+\left(\frac{\bar{C}}{\bar{\delta}}\right)^{n}\right)\right)\left(\bar{A}\|y-\bar{y}\|^2_{\mathcal{S}^{\infty}}+\bar{B}\|(z-\bar{z})\cdot W\|^2_{BMO}\right)
\end{align*}
Therefore, we have
\begin{align*}
&\|Y-\bar{Y}\|^2_{\mathcal{S}^{\infty}}+\|(Z_s-\bar{Z}_s)\cdot W\|^2_{BMO}\\
&\leq\left(n+\frac{\bar{C}n^4+n^3}{\bar{\delta}}\left(1+\left(\frac{\bar{C}}{\bar{\delta}}\right)^{n}\right)\right)\left(\bar{A}\|y-\bar{y}\|^2_{\mathcal{S}^{\infty}}+\bar{B}\|(z-\bar{z})\cdot W\|^2_{BMO}\right)\\
&\leq \frac{1}{2}\left(\|y-\bar{y}\|^2_{\mathcal{S}^{\infty}}+\|(z-\bar{z})\cdot W\|^2_{BMO}\right),
\end{align*}
which implies BSDE \eqref{Eq1} admits a unique solution $(Y,Z)$ such that $(Y,Z\cdot W)\in\mathcal{S}^{\infty}(\mathbb{R}^n)\times BMO$ with $\|Y\|_{\mathcal{S}^{\infty}}\leq A$ and $\|Z\cdot W\|_{BMO}\leq B$.
\end{proof}

Under additional conditions, we have the follow theorem which gives us global solutions.
\begin{theorem}\label{global}
Assume (A1)-(A4) hold and additionally $h^i\leq 0$ and $|l^i|\leq C$ for $i=1,\ldots,n$, then BSDE \eqref{Eq1} admits a unique solution $(Y,Z)$ such that $(Y,Z\cdot W)\in\mathcal{S}^{\infty}(\mathbb{R}^n)\times BMO$.
\end{theorem}
\begin{proof}
Let $\beta_{\cdot}$ be the unique solution of the following ordinary differential equation
\begin{align*}
\beta_t&=n+n\left(1+4n^2\left(1+\left(4Ce^{\|\xi\|_{\infty}}\right)^n\right)\right)\left(e^{\|\xi\|_{\infty}}+\|\xi\|_{\infty}+\left(C^2e^{2\|\xi\|_{\infty}}+2Ce^{\|\xi\|_{\infty}}\right)T\right)\\
&\quad
+\frac{1-\alpha}{8(1-\alpha)}\left(4(1+\alpha)n Ce^{\|\xi\|_{\infty}}\left(1+4n^2\left(1+\left(4Ce^{\|\xi\|_{\infty}}\right)^n\right)\right)\right)^{\frac{2}{1-\alpha}}T\\
&\quad +n Ce^{\|\xi\|_{\infty}}\left(1+4n^2\left(1+\left(4Ce^{\|\xi\|_{\infty}}\right)^n\right)\right)\int_t^T \beta_sds.
\end{align*}
It is easy to see that $\beta_{\cdot}$ is a continuous and decreasing function. Define
\begin{equation*}
\lambda:=\sup_{t\in[0,T]}\beta_t=\beta_0.
\end{equation*}
As $\|\xi\|_{\infty}\leq\lambda$, from Theorem \ref{local}, there exists $\eta_{\lambda}>0$ which only depends on $\lambda$ such that BSDE has a local solution $(Y,Z)$ on $[T-\eta_{\lambda},T]$.
\begin{align*}
Y^1_t&=\xi^1+\int_t^T\frac{1}{2}|Z^1_s|^2+Z^1_sl^1(s,Y_s)+h^1(s,Y_s,Z_s)ds-\int_t^TZ^1_sdW_s\\
&\leq \xi^1+\int_t^T\frac{1}{2}|Z^1_s|^2+Z^1_sl^1(s,Y_s)ds-\int_t^TZ^1_sdW_s
\end{align*}
and
\begin{align*}
Y^i_t&=\xi^i+\int_t^T\frac{1}{2}|Z^i_s|^2+Z^i_sl^i(s,Y_s,Z_s)-k^i(s,Z_s)+h^i(s,Y_s,Z_s)ds-\int_t^TZ^i_sdW_s\\
&\leq \xi^i+\int_t^T\frac{1}{2}|Z^i_s|^2+Z^i_sl^i(s,Y_s,Z_s)-\int_t^TZ^i_sdW_s
\end{align*}
Therefore, we have $Y^i_t\leq\|\xi^i\|_{\infty}$.
Applying It\^{o}'s formula to $u(Y^{1})$, we obtain that
\begin{align*}
&u(Y^{1}_t)\\
&=u(\xi^1)-\int_t^Tu'(Y^{1}_s)Z^{1}_sdW_s\\
&+\int_t^T\left(u'(Y^{1}_s)\left(\frac{1}{2}|Z^1_s|^2+Z^1_sl^1(Y_s)+h^1(s,Y_s,Z_s)\right)-\frac{1}{2}u''(Y^{1}_s)|Z^{1}_s|^2\right)ds\\
&=u(\xi^1)-\int_t^Tu'(Y^{1}_s)Z^{1}_sdW_s+\int_t^T\left(u'(Y^{1}_s)\left(Z^1_sl^1(Y_s)+h^1(s,Y_s,Z_s)\right)-\frac{1}{2}|Z^{1}_s|^2\right)ds\\
&\leq e^{\|\xi^1\|_{\infty}}+\|\xi^1\|_{\infty}-\int_t^Tu'(Y^{1}_s)Z^{1}_sdW_s\\
&\quad +\int_t^T\left(e^{\|\xi^1\|_{\infty}}\left(C|Z^1_s|+C(1+|Y_s|+|Z_s|^{1+\alpha})\right)-\frac{1}{2}|Z^{1}_s|^2\right)ds\\
&\leq e^{\|\xi^1\|_{\infty}}+\|\xi^1\|_{\infty}-\int_t^Tu'(Y^{1}_s)Z^{1}_sdW_s\\
&\quad+\int_t^T\left(\left(C^2e^{2\|\xi^1\|_{\infty}}+Ce^{\|\xi^1\|_{\infty}}(1+|Y_s|+|Z_s|^{1+\alpha})\right)-\frac{1}{4}|Z^{1}_s|^2\right)ds\\
&\leq e^{\|\xi\|_{\infty}}+\|\xi\|_{\infty}-\int_t^Tu'(Y^{1}_s)Z^{1}_sdW_s\\
&\quad+\int_t^T\left(\left(C^2e^{2\|\xi\|_{\infty}}+Ce^{\|\xi\|_{\infty}}(1+|Y_s|+|Z_s|^{1+\alpha})\right)-\frac{1}{4}|Z^{1}_s|^2\right)ds
\end{align*}
Applying It\^{o}'s formula to $u(Y^{i})$, we obtain that
\begin{align*}
&u(Y^{i}_t)\\
&=u(\xi^i)-\int_t^Tu'(Y^{i}_s)Z^{i}_sdW_s\\
&+\int_t^T\left(u'(Y^{i}_s)\left(\frac{1}{2}|Z^i_s|^2+Z^i_sl^i(Y_s,Z_s)-k^i(s,Z_s)+h^i(s,Y_s,Z_s)\right)-\frac{1}{2}u''(Y^{i}_s)|Z^{i}_s|^2\right)ds\\
&=u(\xi^i)-\int_t^Tu'(Y^{i}_s)Z^{i}_sdW_s\\
&+\int_t^T\left(u'(Y^{i}_s)\left(Z^i_sl^i(Y_s,Z_s)-k^i(s,Z_s)+h^i(s,Y_s,Z_s)\right)-\frac{1}{2}|Z^{i}_s|^2\right)ds\\
&\leq e^{\|\xi^i\|_{\infty}}+\|\xi^i\|_{\infty}-\int_t^Tu'(Y^{i}_s)Z^{i}_sdW_s\\
&+\int_t^T\left(e^{\|\xi^i\|_{\infty}}\left(C|Z^i_s|+C\left(1+\sum_{j=1}^{i-1}|Z^{j}_s|^2\right)+C(1+|Y_s|+|Z_s|^{1+\alpha})\right)-\frac{1}{2}|Z^{i}_s|^2\right)ds\\
&\leq e^{\|\xi^i\|_{\infty}}+\|\xi^i\|_{\infty}-\int_t^Tu'(Y^{i}_s)Z^{i}_sdW_s\\
&+\int_t^T\left(C^2e^{2\|\xi^i\|_{\infty}}+Ce^{\|\xi^i\|_{\infty}}\sum_{j=1}^{i-1}|Z^{j}_s|^2+Ce^{\|\xi^i\|_{\infty}}(2+|Y_s|+|Z_s|^{1+\alpha})-\frac{1}{4}|Z^{i}_s|^2\right)ds\\
&\leq e^{\|\xi\|_{\infty}}+\|\xi\|_{\infty}-\int_t^Tu'(Y^{i}_s)Z^{i}_sdW_s\\
&+\int_t^T\left(C^2e^{2\|\xi\|_{\infty}}+Ce^{\|\xi\|_{\infty}}\sum_{j=1}^{i-1}|Z^{j}_s|^2+Ce^{\|\xi\|_{\infty}}(2+|Y_s|+|Z_s|^{1+\alpha})-\frac{1}{4}|Z^{i}_s|^2\right)ds
\end{align*}
Recursively, taking conditional expectation with respect to $\mathcal{F}_r$ for $r\in[T-\eta_{\lambda},t]$, we can show that
\begin{align*}
&E[u(Y^i_t)|\mathcal{F}_r]\\
&\leq \left(1+4n^2\left(1+\left(4Ce^{\|\xi\|_{\infty}}\right)^n\right)\right)\left(e^{\|\xi\|_{\infty}}+\|\xi\|_{\infty}\right)\\
&\quad +E\left[\int_t^T\left(1+4n^2\left(1+\left(4Ce^{\|\xi\|_{\infty}}\right)^n\right)\right)\left(C^2e^{2\|\xi\|_{\infty}}+2Ce^{\|\xi\|_{\infty}}\right)ds\bigg|\mathcal{F}_r\right]\\
&\quad +E\left[\int_t^T\left(1+4n^2\left(1+\left(4Ce^{\|\xi\|_{\infty}}\right)^n\right)\right)\left(Ce^{\|\xi\|_{\infty}}(|Y_s|+|Z_s|^{1+\alpha})\right)-\frac{1}{4}|Z^{i}_s|^2ds\bigg|\mathcal{F}_r\right]
\end{align*}
Hence, it holds that
\begin{align*}
&E\left[\sum_{i=1}^{n}u(Y^i_t)\bigg|\mathcal{F}_r\right]\\
&\leq n\left(1+4n^2\left(1+\left(4Ce^{\|\xi\|_{\infty}}\right)^n\right)\right)\left(e^{\|\xi\|_{\infty}}+\|\xi\|_{\infty}\right)\\
&\quad +nE\left[\int_t^T\left(1+4n^2\left(1+\left(4Ce^{\|\xi\|_{\infty}}\right)^n\right)\right)\left(C^2e^{2\|\xi\|_{\infty}}+2Ce^{\|\xi\|_{\infty}}\right)ds\bigg|\mathcal{F}_r\right]\\
&\quad +E\left[\int_t^Tn\left(1+4n^2\left(1+\left(4Ce^{\|\xi\|_{\infty}}\right)^n\right)\right)\left(Ce^{\|\xi\|_{\infty}}(|Y_s|+|Z_s|^{1+\alpha})\right)-\frac{1}{4}|Z_s|^2ds\bigg|\mathcal{F}_r\right]\\
&\leq n\left(1+4n^2\left(1+\left(4Ce^{\|\xi\|_{\infty}}\right)^n\right)\right)\left(e^{\|\xi\|_{\infty}}+\|\xi\|_{\infty}\right)\\
&\quad +nE\left[\int_t^T\left(1+4n^2\left(1+\left(4Ce^{\|\xi\|_{\infty}}\right)^n\right)\right)\left(C^2e^{2\|\xi\|_{\infty}}+2Ce^{\|\xi\|_{\infty}}\right)ds\bigg|\mathcal{F}_r\right]\\
&\quad +E\left[\int_t^Tn Ce^{\|\xi\|_{\infty}}\left(1+4n^2\left(1+\left(4Ce^{\|\xi\|_{\infty}}\right)^n\right)\right)|Y_s|ds\bigg|\mathcal{F}_r\right]\\
&\quad
+E\left[\int_t^T\frac{1-\alpha}{8(1-\alpha)}\left(4(1+\alpha)n Ce^{\|\xi\|_{\infty}}\left(1+4n^2\left(1+\left(4Ce^{\|\xi\|_{\infty}}\right)^n\right)\right)\right)^{\frac{2}{1-\alpha}}ds\bigg|\mathcal{F}_r\right]\\
&\quad
-E\left[\int_t^T\frac{1}{8}|Z_s|^2ds\bigg|\mathcal{F}_r\right]
\end{align*}
Noting that
\begin{equation*}
u(Y^i_t)\geq |Y^i_t|-1,
\end{equation*}
we have
\begin{align*}
&E\left[|Y_t||\mathcal{F}_r\right]+\frac{1}{8}E\left[\int_t^T|Z_s|^2\bigg|\mathcal{F}_r\right]\\
&\leq n+n\left(1+4n^2\left(1+\left(4Ce^{\|\xi\|_{\infty}}\right)^n\right)\right)\left(e^{\|\xi\|_{\infty}}+\|\xi\|_{\infty}+\left(C^2e^{2\|\xi\|_{\infty}}+2Ce^{\|\xi\|_{\infty}}\right)T\right)\\
&\quad
+\frac{1-\alpha}{8(1-\alpha)}\left(4(1+\alpha)n Ce^{\|\xi\|_{\infty}}\left(1+4n^2\left(1+\left(4Ce^{\|\xi\|_{\infty}}\right)^n\right)\right)\right)^{\frac{2}{1-\alpha}}T\\
&\quad +n Ce^{\|\xi\|_{\infty}}\left(1+4n^2\left(1+\left(4Ce^{\|\xi\|_{\infty}}\right)^n\right)\right)\int_t^T E\left[|Y_s||\mathcal{F}_r\right]ds.
\end{align*}
Hence, we deduce that
\begin{equation*}
E\left[|Y_t||\mathcal{F}_r\right]\leq\beta_t.
\end{equation*}
Setting $r=t$, we have
\begin{equation*}
|Y_t|\leq\beta_t\leq\lambda.
\end{equation*}
Taking $T-\eta_{\lambda}$ as the terminal time and $Y_{T-\eta_{\lambda}}$ as terminal value, from Theorem \ref{local}, BSDE \eqref{Eq1} has a local solution $(Y,Z)$ on $[T-2\eta_{\lambda},T-\eta_{\lambda}]$. Once again, we can deduce that $|Y_t|\leq\beta_t$, for $t\in[T-2\eta_{\lambda},T-\eta_{\lambda}]$. Repeating the preceding process, we can extend the pair $(Y,Z)$ to the whole interval $[0,T]$ within a finite steps such that $Y$ is uniformly bounded by $\lambda$. Moreover, similar to the above, we have
\begin{align*}
\frac{1}{8}E\left[\int_t^T|Z_s|^2|\mathcal{F}_t\right]\leq \lambda.
\end{align*}
Consequently, we have
\begin{equation*}
\|Z\cdot W\|^2_{BMO}\leq 8\lambda.
\end{equation*}
Finally, the uniqueness on the given interval $[0,T]$ follows from Theorem \ref{local} and a pasting technique.
\end{proof}
\section{Triangularly quadratic BSDEs with path dependence in value process}

Based on the arguments in the above section, we are able to consider the following type of triangularly quadratic BSDEs with path dependence in value process:
\begin{align}\label{Eq2}
\begin{cases}
&Y^{1}_{t}=\xi^1+\int_t^T\left(\frac{1}{2}|Z^1_s|^2+Z^1_sl^1(s,G^{1}_s(Y))+h^1(s,G^{1}_s(Y),Z_s)\right)ds-\int_t^TZ^{1}_sdW_s,\\
&Y^{i}_{t}=\xi^i+\int_t^T\left(\frac{1}{2}|Z^i_s|^2+Z^{i}_sl^i(s,G^{i}_s(Y),Z_s)-k^i(s,Z_s)+h^i(s,G^{i}_s(Y),Z_s)\right)ds\\
&\quad \quad \quad -\int_t^TZ^{i}_sdW_s,\quad,\quad i=2,\ldots,n.
\end{cases}
\end{align}
where $\xi$ is $\mathbb{R}^d$-valued and $\mathcal{F}_T$-measurable random variable which is bounded and for any $0\leq t\leq T$, $G^{i}_t:\mathcal{C}_T(\mathbb{R}^n)\rightarrow \mathbb{R}^n$ is a function for $i=1,\ldots,n$. We suppose that the sequence of functions $G^i$, $i=1,\ldots,n$ satisfy the following assumption.
\begin{itemize}
\item[(A5)] For any $0\leq t\leq T$ and $y,\bar{y}\in\mathcal{C}_T(\mathbb{R}^n)$, we have $G^{i}_t(0)=0$ and
    \begin{align*}
    &G^{i}_t(y)=G^{i}_t(\{y_{s\wedge t}\}_{0\leq s\leq t\leq T}),\\
    &|G^{i}_t(y)-G^{i}_t(\bar{y})|\leq\sup_{0\leq u\leq t}|y_u-\bar{y}_u|,
    \end{align*}
\end{itemize}
\begin{remark}
Given some $\epsilon\geq 0$, typical examples satisfying assumption (A5) are
\begin{itemize}
\item[$\bullet$] the delayed value of the solution, $G^{i}_t:y\mapsto y_{(t-\epsilon)^+}$;
\item[$\bullet$] the recent maximum of the solution, $G^{i}_t:y\mapsto \sup_{(t-\epsilon)^+\leq u\leq t}y_u$;
\item[$\bullet$] the averaged recent value of the solution, $G^{i}_t:y\mapsto \frac{1}{\epsilon}\int^t_{(t-\epsilon)^+}y_udu$;
\item[$\bullet$] the cumulated recent value of the solution, $G^{i}_t:y\mapsto \int^t_{(t-\epsilon)^+}y_udu$ for $\epsilon\leq 1$ or $T\leq 1$.
\end{itemize}
\end{remark}
\begin{theorem}
Assume (A1)-(A5) hold, then there exist constants $T_{\eta}$, $C_1$ and $C_2$ only depending on $\alpha$, $C$ and $\|\xi\|_{\infty}$ such that for $T\leq T_{\eta}$, BSDE \eqref{Eq2} admits a unique solution $(Y,Z)$ such that $(Y,Z\cdot W)\in\mathcal{S}^{\infty}(\mathbb{R}^n)\times BMO$ with $\|Y\|_{\mathcal{S}^{\infty}}\leq C_1$ and $\|Z\cdot W\|_{BMO}\leq C_2$.
\end{theorem}
\begin{proof}
The argument follows from a similar technique as in the proof of Theorem \ref{local}.
\end{proof}
\begin{remark}
For $T\leq 1$ and $G^i_t(y)=\int_0^ty_sds$ for $i=1,\ldots,n$, BSDE \eqref{Eq2} is equivalent to the following FBSDE
\begin{align*}
\begin{cases}
&X_t=\int_0^tY_sds,\\
&Y^{1}_{t}=\xi^1+\int_t^T\left(\frac{1}{2}|Z^1_s|^2+Z^1_sl^1(s,X_s)+h^1(s,X_s,Z_s)\right)ds-\int_t^TZ^{1}_sdW_s,\\
&Y^{i}_{t}=\xi^i+\int_t^T\left(\frac{1}{2}|Z^i_s|^2+Z^{i}_sl^i(s,X_s,Z_s)-k^i(s,Z_s)+h^i(s,X_s,Z_s)\right)ds\\
&\quad \quad \quad -\int_t^TZ^{i}_sdW_s,\quad,\quad i=2,\ldots,n,
\end{cases}
\end{align*}
which in general admits a solution only for a small time horizon. We refer to \cite{KLT, LT} for related studies.
\end{remark}
For a given delay parameter $\epsilon>0$, we consider the following BSDE with delay in value process:
\begin{align}\label{Eq3}
Y^{i}_{t}=\xi^i+\int_t^T\left(\frac{1}{2}|Z^i_s|^2+Z^{i}_sl^i(s,G^{i,\epsilon}_s(Y))+h^i(s,G^{i,\epsilon}_s(Y),Z_s)\right)ds-\int_t^TZ^{i}_sdW_s,
\end{align}
where $\xi$ is $\mathbb{R}^d$-valued and $\mathcal{F}_T$-measurable random variable which is bounded and for any $0\leq t\leq T$, $G^{i,\epsilon}_t:\mathcal{C}_T(\mathbb{R}^n)\rightarrow \mathbb{R}^n$ is a function for $i=1,\ldots,n$. We suppose that the sequence of functions $G^{i,\epsilon}$, $i=1,\ldots,n$ satisfy the following assumption.
\begin{itemize}
\item[(A6)] For any $0\leq t\leq T$, $\epsilon>0$ and $y,\bar{y}\in\mathcal{C}_T(\mathbb{R}^n)$, we have $G^{i,\epsilon}_t(0)=0$ and
    \begin{align*}
    &G^{i,\epsilon}_t(y)=G^{i,\epsilon}_t(\{y_{s\wedge t}\}_{0\leq s\leq t\leq T});\\
    &|G^{i,\epsilon}_t(y)-G^{i,\epsilon}_t(\bar{y})|\leq \sup_{(t-\epsilon)^+\leq u\leq t}|y_u-\bar{y}_u|.
    \end{align*}
\end{itemize}
We also make the following assumptions:
\begin{itemize}
\item[(A7)]
    For $i=1,\ldots,n$, $l^i:\Omega\times[0,T]\times\mathbb{R}^n\rightarrow\mathbb{R}$ satisfies that $l^i(\cdot,y)$ is adapted for each $y\in\mathbb{R}^n$. Moreover, it holds that
    \begin{align*}
    &|l^i(t,y)|\leq C,~~y\in\mathbb{R}^n;\\
    &|l^i(t,y)-l^i(t,\bar{y})|\leq C|y-\bar{y}|+C\sum_{j=1}^{i-1}|z^j-\bar{z}^j|,~~~y,\bar{y}\in\mathbb{R}^n;
    \end{align*}
\item[(A8)] $h:\Omega\times[0,T]\times\mathbb{R}^n\times\mathbb{R}^{n\times d}\rightarrow\mathbb{R}^n$ satisfies that $h(\cdot,y,z)$ is adapted for each $y\in\mathbb{R}^n$ and $z\in\mathbb{R}^{n\times d}$. Moreover, there exists $\alpha\in[-1,1)$ such that
    \begin{align*}
    &-C(1+|z|^{1+\alpha})\leq h(t,y,z)\leq 0,~~y\in\mathbb{R}^n,~z\in\mathbb{R}^{n\times d};\\
    &|h(t,y,z)-h(t,\bar{y},\bar{z})|\leq C|y-\bar{y}|+C\left(1+|z|^{\alpha^+}+|\bar{z}|^{\alpha^+}\right)|z-\bar{z}|,
    \end{align*}
    for $y,\bar{y}\in\mathbb{R}^n$ and $z,\bar{z}\in\mathbb{R}^{n\times d}$.
\end{itemize}
We denote
\begin{align*}
&\bar{\beta}=n+n\left(1+4n^2\left(1+\left(4Ce^{\|\xi\|_{\infty}}\right)^n\right)\right)\left(e^{\|\xi\|_{\infty}}+\|\xi\|_{\infty}+\left(C^2e^{2\|\xi\|_{\infty}}+2Ce^{\|\xi\|_{\infty}}\right)T\right)\\
&\quad
+\frac{1-\alpha}{8(1-\alpha)}\left(4(1+\alpha)n Ce^{\|\xi\|_{\infty}}\left(1+4n^2\left(1+\left(4Ce^{\|\xi\|_{\infty}}\right)^n\right)\right)\right)^{\frac{2}{1-\alpha}}T,\\
&\tilde{\Delta}=\Delta(C\sqrt{2T}+16\sqrt{2}\bar{\beta}),\\
&\tilde{\delta}=\delta(C\sqrt{2T}+16\sqrt{2}\bar{\beta}),\\
&\epsilon_0=\frac{\tilde{\delta}}{4nC^2\left(2e\tilde{\delta}T+2e\tilde{\Delta}\tilde{\delta}\bar{\beta}+1\right)}.
\end{align*}
\begin{theorem}
Assume (A6)-(A8) hold. Then for any $0<\epsilon\leq\epsilon_0$, BSDE \eqref{Eq3} admits a unique solution $(Y, Z)$ such that $(Y,Z\cdot W)\in\mathcal{S}^{\infty}(\mathbb{R}^n)\times BMO$.
\end{theorem}
\begin{proof}
For $y\in\mathcal{S}^{\infty}(\mathbb{R}^n)$, if follows from Theorem \ref{global} that the following BSDE
\begin{align*}
Y^{i}_{t}=\xi^i+\int_t^T\left(\frac{1}{2}|Z^i_s|^2+Z^{i}_sl^i(s,G^{i,\epsilon}_s(y))+h^i(s,G^{i,\epsilon}_s(y),Z_s)\right)ds-\int_t^TZ^{i}_sdW_s
\end{align*}
admits a unique solution $(Y,Z)$ such that $(Y,Z\cdot W)\in\mathcal{S}^{\infty}(\mathbb{R}^n)\times BMO$ and
\begin{equation*}
\|Y\|_{\mathcal{S}^\infty}\leq\bar{\beta},~~~~\|Z\cdot W\|^2_{BMO}\leq 8\bar{\beta}.
\end{equation*}
Let us introduce the function $\varphi:\mathcal{S}^{\infty}(\mathbb{R}^n)\mapsto\mathcal{S}^{\infty}(\mathbb{R}^n)$ such that for any $y\in\mathcal{S}^{\infty}(\mathbb{R}^n)$, $\phi(y)=Y$ is the first component to the following BSDE
\begin{align*}
Y^{i}_{t}=\xi^i+\int_t^T\left(\frac{1}{2}|Z^i_s|^2+Z^{i}_sl^i(s,G^{i,\epsilon}_s(y))+h^i(s,G^{i,\epsilon}_s(y),Z_s)\right)ds-\int_t^TZ^{i}_sdW_s
\end{align*}
We consider $(y,\bar{y})\in\mathcal{S}^{\infty}(\mathbb{R}^n)\times\mathcal{S}^{\infty}(\mathbb{R}^n)$ and denote by $(Y,Z)$ and $(\bar{Y},\bar{Z})$ the solutions to the associated BSDEs. From the above, we have
\begin{equation*}
\|Z\cdot W\|^2_{BMO}\leq 8\bar{\beta},~~~\|\bar{Z}\cdot W\|^2_{BMO}\leq 8\bar{\beta}.
\end{equation*}
Moreover, applying It\^{o}'s formula to $e^{\gamma t}|Y^i_t-\bar{Y}^i_t|^2$, we have
\begin{align*}
&e^{\gamma t}|Y^i_t-\bar{Y}^i_t|^2+\int_t^Te^{\gamma s}|Z^i_s-\bar{Z}^i_s|^2ds\\
&=2\int_t^Te^{\gamma s}\left(Y^i_s-\bar{Y}^i_s\right)\left(\frac{1}{2}|Z^i_s|^2-\frac{1}{2}|\bar{Z}^i_s|^2+Z^i_sl^i(s,G^{i,\epsilon}_s(y))-\bar{Z}^i_sl^i(s,G^{i,\epsilon}_s(\bar{y}))\right)ds\\
&\quad +2\int_t^Te^{\gamma s}\left(Y^i_s-\bar{Y}^i_s\right)\left(h^i(s,G^{i,\epsilon}_s(y),Z_s)-h^i(s,G^{i,\epsilon}_s(\bar{y}),\bar{Z}_s)\right)ds\\
&\quad -\gamma\int_t^Te^{\gamma s}|Y^i_s-\bar{Y}^i_s|^2ds-2\int_t^Te^{\gamma s}\left(Y^i_s-\bar{Y}^i_s\right)\left(Z^i_s-\bar{Z}^i_s\right)dW_s\\
&=2\int_t^Te^{\gamma s}\left(Y^i_s-\bar{Y}^i_s\right)\bar{Z}^i_s\left(l^i(s,G^{i,\epsilon}_s(y))-l^i(s,G^{i,\epsilon}_s(\bar{y}))\right)ds\\
&\quad+2\int_t^Te^{\gamma s}\left(Y^i_s-\bar{Y}^i_s\right)\left(h^i(s,G^{i,\epsilon}_s(y),Z_s)-h^i(s,G^{i,\epsilon}_s(\bar{y}),\bar{Z}_s)\right)ds\\
&\quad -\gamma\int_t^Te^{\gamma s}|Y^i_s-\bar{Y}^i_s|^2ds-2\int_t^Te^{\gamma s}\left(Y^i_s-\bar{Y}^i_s\right)\left(Z^i_s-\bar{Z}^i_s\right)d\bar{W}^i_s,
\end{align*}
where $\bar{W}^i_t=W_t-\int_{0}^t\left(\frac{1}{2}(Z^i_s+\bar{Z}^i_s)+l^i(s,G^{i,\epsilon}_s(y))\right)ds$ is a Brownian motion under the equivalent probability measure $\bar{P}^i$ defined by
\begin{equation*}
\frac{d\bar{P}^i}{dP}=\mathcal{E}_T\left(\left(\frac{1}{2}\left(Z^i_{\cdot}+\bar{Z}^i_{\cdot}\right)+l^i(\cdot,G^{i,\epsilon}_{\cdot}(y))\right)\cdot W\right).
\end{equation*}
Therefore, we have
\begin{align*}
&e^{\gamma t}|Y^i_t-\bar{Y}^i_t|^2+\int_t^Te^{\gamma s}|Z^i_s-\bar{Z}^i_s|^2ds\\
&\leq2C\int_t^Te^{\gamma s}|Y^i_s-\bar{Y}^i_s||\bar{Z}^i_s||G^{i,\epsilon}_s(y)-G^{i,\epsilon}_s(\bar{y})|ds\\
&\quad+2C\int_t^Te^{\gamma s}|Y^i_s-\bar{Y}^i_s|\left(|G^{i,\epsilon}_s(y)-G^{i,\epsilon}_s(\bar{y})|+|Z_s-\bar{Z}_s|\right)ds\\
&\quad -\gamma\int_t^Te^{\gamma s}|Y^i_s-\bar{Y}^i_s|^2ds-2\int_t^Te^{\gamma s}\left(Y^i_s-\bar{Y}^i_s\right)\left(Z^i_s-\bar{Z}^i_s\right)d\bar{W}^i_s\\
&\leq \left(8nC^2Te^{\gamma\epsilon}+\frac{4n\tilde{\Delta}C^2}{\tilde{\delta}}+8nC^2\tilde{\Delta}\bar{\beta}e^{\gamma\epsilon}\right)\int_t^Te^{\gamma s}|Y^i_s-\bar{Y}^i_s|^2ds-\gamma \int_t^Te^{\gamma s}|Y^i_s-\bar{Y}^i_s|^2ds\\
&\quad+\frac{1}{64n\tilde{\Delta}\bar{\beta}e^{\gamma\epsilon}}\int_t^Te^{\gamma s}|\bar{Z}^i_s|^2|G^{i,\epsilon}_s(y)-G^{i,\epsilon}_s(\bar{y})|^2ds+\frac{1}{8nTe^{\gamma\epsilon}}\int_t^Te^{\gamma s}|G^{i,\delta}_s(y)-G^{i,\delta}_s(\bar{y})|^2ds\\
&\quad+\frac{\tilde{\delta}}{4n\tilde{\Delta}}\int_t^Te^{\gamma s}|Z_s-\bar{Z}_s|^2ds-2\int_t^Te^{\gamma s}\left(Y^i_s-\bar{Y}^i_s\right)\left(Z^i_s-\bar{Z}^i_s\right)d\bar{W}^i_s.
\end{align*}
Since $\epsilon\leq\epsilon_0$, by letting $\lambda=\frac{1}{\epsilon}$, we have
\begin{align*}
&e^{\gamma t}|Y^i_t-\bar{Y}^i_t|^2+\int_t^Te^{\gamma s}|Z^i_s-\bar{Z}^i_s|^2ds\\
&\leq\frac{1}{64n\tilde{\Delta}\bar{\beta}e^{\gamma\epsilon}}\int_t^Te^{\gamma s}|\bar{Z}^i_s|^2|G^{i,\epsilon}_s(y)-G^{i,\epsilon}_s(\bar{y})|^2ds+\frac{1}{8nTe^{\gamma\epsilon}}\int_t^Te^{\gamma s}|G^{i,\delta}_s(y)-G^{i,\delta}_s(\bar{y})|^2ds\\
&\quad+\frac{\tilde{\delta}}{4n\tilde{\Delta}}\int_t^Te^{\gamma s}|Z_s-\bar{Z}_s|^2ds-2\int_t^Te^{\gamma s}\left(Y^i_s-\bar{Y}^i_s\right)\left(Z^i_s-\bar{Z}^i_s\right)d\bar{W}^i_s.
\end{align*}
Noting that for $\epsilon>0$, we have
\begin{align*}
e^{\gamma s}|G^{i,\epsilon}_s(y)-G^{i,\epsilon}_s(\bar{y})|^2&\leq e^{\gamma s}\sup_{(s-\epsilon)^+\leq r\leq s}|y_r-\bar{y}_r|^2\\
&\leq e^{\gamma \epsilon}\sup_{(s-\epsilon)^+\leq r\leq s}e^{\gamma r}|y_r-\bar{y}_r|^2ds\\
&\leq e^{\gamma \epsilon}\|e^{\gamma_{\cdot}}|y-\bar{y}|^2\|_{\mathcal{S}^{\infty}}.
\end{align*}
Hence we have
\begin{align*}
&e^{\gamma t}|Y^i_t-\bar{Y}^i_t|^2+\int_t^Te^{\gamma s}|Z^i_s-\bar{Z}^i_s|^2ds\\
&\leq\frac{1}{64n\tilde{\Delta}\bar{\beta}}\|e^{\gamma_{\cdot}}|y-\bar{y}|^2\|_{\mathcal{S}^{\infty}}\int_t^T|\bar{Z}^i_s|^2ds+\frac{1}{8n}\|e^{\gamma_{\cdot}}|y-\bar{y}|^2\|_{\mathcal{S}^{\infty}}\\
&\quad+\frac{\tilde{\delta}}{4n\tilde{\Delta}}\int_t^Te^{\gamma s}|Z_s-\bar{Z}_s|^2ds-2\int_t^Te^{\gamma s}\left(Y^i_s-\bar{Y}^i_s\right)\left(Z^i_s-\bar{Z}^i_s\right)d\bar{W}^i_s.
\end{align*}
Therefore, it follows from Lemma \ref{PP} that
\begin{align*}
&\|e^{\gamma \cdot}|Y^i-\bar{Y}^i|^2\|_{\mathcal{S}^{\infty}}+\tilde{\delta}\|e^{\frac{\gamma \cdot}{2}}(Z^i-\bar{Z}^i)\cdot W\|^2_{BMO}\\
&\leq\frac{1}{32n\tilde{\Delta}\bar{\beta}}\|e^{\gamma_{\cdot}}|y-\bar{y}|^2\|_{\mathcal{S}^{\infty}}\tilde{\Delta}\|\bar{Z}^i\cdot W\|^2_{BMO}+\frac{1}{4n}\|e^{\gamma_{\cdot}}|y-\bar{y}|^2\|_{\mathcal{S}^{\infty}}\\
&\quad+\frac{\tilde{\delta}}{2n\tilde{\Delta}}\tilde{\Delta}\|e^{\frac{\gamma \cdot}{2}}(Z-\bar{Z})\cdot W\|^2_{BMO}\\
&\leq\frac{1}{2n}\|e^{\gamma_{\cdot}}|y-\bar{y}|^2\|_{\mathcal{S}^{\infty}}+\frac{\tilde{\delta}}{2n}\|e^{\frac{\gamma \cdot}{2}}(Z-\bar{Z})\cdot W\|^2_{BMO}.
\end{align*}
Finally, we deduce that
\begin{align*}
\|e^{\gamma_{\cdot}}|Y-\bar{Y}|^2\|_{\mathcal{S}^{\infty}}&\leq \frac{1}{2}\|e^{\gamma_{\cdot}}|y-\bar{y}|^2\|_{\mathcal{S}^{\infty}}.
\end{align*}
Thus, $\varphi$ is a contraction on $\mathcal{S}^{\infty}(\mathbb{R}^n)$ so that there exists a unique solution $(Y,Z)$ to BSDE \eqref{Eq3} such that $(Y,Z\cdot W)\in\mathcal{S}^{\infty}(\mathbb{R}^n)\times BMO$.
\end{proof}                                                                                           %


\begin{thebibliography}{99}                                                                                               %
\bibitem{B} J.M. Bismut. Conjugate convex functions in optimal stochastic control. \emph{J. Math. Anal. Appl.} \textbf{44}, (1973), 384--404.
\bibitem{BEH} K. Bahlali, E. H. Essaky, M. Hassani. Multidimensional BSDEs with super-linear growth coefficient: Application to degenerate systems of semilinear PDEs. \emph{C.R. Acad. Sci. Paris S\'er. I Math.} \textbf{348}, (2010), 677--682.
\bibitem{BE} P. Briand and R. Elie. A simple constructive approach to quadratic BSDEs with or without delay. \emph{Stochastic Processes and their Applications}. \textbf{123(8)}, 2013, 2921-2939.
\bibitem{BH} P. Briand and Y. Hu. BSDE with quadratic growth and unbounded terminal value. \emph{Probability Theory and Related Fields.} \textbf{136}, (2006), 604--618.
\bibitem{BH1} P. Briand and Y. Hu. Quadratic BSDEs with convex generators and unbounded terminal conditions. \emph{Probability Theory and Related Fields.} \textbf{141}, (2008), 543--567.
\bibitem{CN1} P. Cheridito and K. Nam. BSDEs with terminal conditions that have bounded Malliavin derivative. \emph{Journal of Functional Analysis} \textbf{266}, (2014), 1257--1285.
\bibitem{CN} P. Cheridito and K. Nam. Multidimensional quadratic and subquadratic BSDEs with special structure. \emph{Stochastics} \textbf{87}, (2015), 871--884.
\bibitem{CM} B. Chikvinidze and M. Mania. New proofs of some results on bounded mean oscillation martingales
using Backward stochastic differential equations. \emph{J. Theor. Probab.} \textbf{27}, (2014), 1213--1228.
\bibitem{DHB} F. Delbaen, Y. Hu, X. Bao. Backward SDEs with superquadratic growth. \emph{Probability Theory and Related Fields.} \textbf{150}, (2011), 145--192.
\bibitem{EH} N. El Karoui and S. Hamad\`{e}ne. BSDEs and risk-sensitive control, zero-sum and non-zero sum game problems of stochastic functional differential equations. \emph{Stochastic Process. Appl.} \textbf{107}, (2003), 145--169.
\bibitem{F} C. Frei. Splitting multidimensional BSDEs and finding local equilibria. \emph{Stochastic Process. Appl.} \textbf{124}, (2014), 2654--2671.
\bibitem{FR} C. Frei and dos Reis. A financial market with interacting investors: does an equilibrium exist? \emph{Math. Financ. Econ.} \textbf{4}, (2011), 161--182.
\bibitem{HR} J. Harter and A. Richou. A stability approach for solving multidimensional quadaratic BSDEs. \emph{Electron. J. Probab.} \textbf{24(4)}, 2019, 1-51.
\bibitem{HT} Y. Hu and S. Tang. Multi-dimensional backward stochastic differential equations of diagonally quadratic generators. \emph{Stochastic Process. Appl.} \textbf{126}, (2016), 1066--1086.
\bibitem{HT1} Y. Hu and S. Tang. Non-zero sum quadratic differential game of BSDEs and multi-dimensional diagonally quadratic BSDE. \emph{IFAC-PapersOnline} \textbf{49}, (2016), 308--309.
\bibitem{JKL} A. Jamneshan, M. Kupper and P. Luo. Multidimensional quadratic BSDEs with separated generators. \emph{Electronic Communications in Probability}. \textbf{22(58)}, 2017, 1-10.
\bibitem{KXZ} C. Kardaras, H. Xing and G. \v{Z}itkovi\'{c}. Incomplete stochastic equilibria with exponential utilities close to Pareto optimality, arXiv:1505.07224v1
\bibitem{Ka} N. Kazamaki. Continuous Exponential Martingale and BMO. Lecture Notes in Mathematics, vol. 1579. \emph{Springer-Verlag}, Berlin, 1994. viii+91 pp.
\bibitem{Ko} M. Kobylanski. Backward stochastic differential equations and partial differential equations with quadratic growth. \emph{Annals of Probability.} \textbf{28}, (2000), 558--602.
\bibitem{KP} D. Kramkov and S. Pulido. A system of quadratic BSDEs arising in a price impact model. \emph{Ann. Appl. Probab.} \textbf{26}, (2016), 794--817.
\bibitem{KLT} M. Kupper, P. Luo and L. Tangpi. Multidimensional Markovian FBSDEs with super-quadratic growth. \emph{Stochastic Processes and their Applications}. \textbf{129(3)}, 2019, 902-923.
\bibitem{LT} P. Luo and L. Tangpi. Solvability of coupled FBSDEs with diagonally quadratic generators. \emph{Stochastics and Dynamics}. \textbf{17(6)}, (2017), 1750043.
\bibitem{MR} F. Masiero and A. Richou. A note on the existence of solutions to Markovian superquadratic BSDEs with an unbounded ternimal condtion. \emph{Electron. J. Probab} \textbf{18}, (2013), 1-15.
\bibitem{PP} E.Pardoux and S. G. Peng. Adapted solution of a backward stochastic differential equation. \emph{System Control Lett.} \textbf{14}, (1990), 55--61.
\bibitem{R} A. Richou. Markovian quadratic and superquadratic BSDEs with an unbounded terminal condition. \emph{Stochastic Process. Appl.} \textbf{122}, (2012), 3173--3208.
\bibitem{T} R. Tevzadze. Solvability of backward stochastic differential equations with quadratic growth. \emph{Stochastic Process. Appl.} \textbf{118}, (2008), 503--515.
\bibitem{NT} N. Touzi. Optimal Stochastic Control, Stochastic Target Problems, and Backward SDE. Fields Institute Monographs, vol. 29. \emph{Springer}, New York, 2013. x+214 pp.
\bibitem{XZ} H. Xing and G. \v{Z}itkovi\'{c}. A class of globally solvable Markovian quadratic BSDE systems and applications. \emph{Ann. Probab.} \textbf{46(1)}, 2018, 491-550.
\end{thebibliography}
\end{document}